\documentclass[12pt]{article}
\usepackage{epsfig}
\usepackage{graphicx}
\usepackage{amsthm}
\usepackage{amsfonts}
\usepackage{color}
\usepackage{hyperref}
\usepackage{amsmath}
\usepackage[authoryear]{natbib}
\newcommand{\field}[1]{\mathbb{#1}}
\newcommand{\R}{\field{R}}

\newcommand{\p}{\field{P}}
\newcommand{\N}{\field{N}}

\newcommand{\E}{\field{E}}
\newcommand{\FF}{\mathcal{F}}

\newcommand{\bPhi}{\mbox{\boldmath$\Phi$}}

\newcommand{\epsilonv}{\mbox{\boldmath$\epsilon$}}

\def\argmin{\mathop{\mbox{argmin}}}
\theoremstyle{Conjecture} \theoremstyle{example}
\theoremstyle{remark} \theoremstyle{lemma}
\theoremstyle{definition} \theoremstyle{corol}
\theoremstyle{proposition} \theoremstyle{condition}
\newtheorem{theorem}{Theorem}[section]
\newtheorem{example}{Example}[section]

\newtheorem{lemma}{Lemma}[section]
\newtheorem{definition}{Definition}[section]

\newtheorem{proposition}{Proposition}[section]

\def\lf{\lfloor}
\def\rf{\rfloor}

\def\bX{{\bf X}}

\def\bS{{\bf S}}
\def\ba{{\bf a}}
\def\bA{{\bf A}}
\def\bN{{\bf N}}
\def\bT{{\bf T}}
\def\bg{{\bf g}}
\def\bh{{\bf h}}
\def\be{{\bf e}}
\def\bW{{\bf W}}

\begin{document}
\centerline{\large\bf Statistical Inference for High Dimensional Panel Functional Time Series}

	\centerline{\sc Zhou Zhou
		and Holger Dette}
		
	\centerline{\sc \small University of Toronto and Ruhr-Universit\"at Bochum  }

\bigskip
\font\n=cmcsc10

\centerline{\today}

\begin{abstract}
In this paper we develop statistical inference tools for high dimensional functional time series.
 We introduce  a new concept  of physical dependent processes  in the space of square integrable functions, which
adopts the idea of basis decomposition of functional data  in these spaces,  and derive Gaussian   
 and multiplier bootstrap approximations for sums of high dimensional
functional time series. These results have numerous  important  statistical consequences. Exemplarily, we consider
the
development of joint  simultaneous confidence bands  for the mean functions  and the construction of
tests for  the hypotheses that the mean  functions in the spatial dimension are  parallel. The results are illustrated by means of a small simulation study and in the analysis of
Canadian temperature data.

	\end{abstract}
	
	\bigskip
	
		

\section{Introduction} \label{sec1}
\def\theequation{1.\arabic{equation}}
\setcounter{equation}{0}

In many fields of  statistics  data, say $X_{i,j}$,  are recorded in time  (denoted by the index $i$)   at different locations (denoted by the index $j$). This type of data
is called panel  or spatio-temporal data and  appears in numerous  applications. For example, in economics data are often collected   for different  firms  or  regions over several years, or in  geo-statistics
and  climate research  data  are recorded over a period of time at different locations.  Statistical methods for
this type of  data are meanwhile  well developed, in particular in the field of econometrics and spatial statistics,
and we refer to the monographs of  \cite{hsiao2003}, \cite{baltagi2005}, \cite{wooldridge2010}, \cite{cressie2015} and \cite{hainingli2020}
  among many others.  In recent years there has also been substantial interest
  in statistical inference tools for high dimensional time series   and we mention  exemplarily the  work of  \cite{cresjoha2008},
    \cite{GALVAO2010},  \cite{belloni}, \cite{banerjee2017}  and \cite{kock_tang_2019}
who developed methodology for various  high dimensional spatio-temporal  models.

A common feature of most of the literature on statistical methodology in  this context  consists in the fact that the data $X_{i,j}$ 
are  real valued  or  multivariate. However in  many modern applications
 more and more data are recorded  continuously  during  a  time  interval  or  intermittently  at  a very dense grid of   time  points.
In such  circumstances it is often reasonable to  use a high-dimensional functional time series model, where the  data  $X_{i,j}$  observed  at time $i$ in panel $j$ is a function.  Throughout this paper we call  this type of data functional panel data or high-dimensional functional time series.
For example,   $X_{i,j}$ could represent the (smoothed) curve of the  price of a stock for a day, where $i$ denotes the  index for the different days and $j$ is  the index for the different stocks. For another example, $X_{i,j}$ could represent the (smoothed) temperature or precipitation curve for a year $i$  at a location $j$.

 Although functional data  analysis  is nowadays  a rather  well developed and very  broad field [see for example the monographs of  \cite{bosq2000}, \cite{RamsaySilverman2005}, \cite{FerratyVieu2010}, \cite{HorvathKokoskza2012} and \cite{hsingeubank2015}],
there does not exist much literature  for panels of functional data, in particular for high  dimensional panels.
Several authors have developed statistical methodology for spatio-temporal data [see \cite{delicado2010}, \cite{gromkokoreim2017}, \cite{kokoreim2019} among others], but this literature usually does not consider the high-dimensional case, where the spatial dimension is increasing with the sample size.
High dimensional  functional  time series have  been investigated  in the context of forecasting and factor analysis [see   
\cite{GAO2019232} or \cite{nistavhal2019}].
Other authors considered statistical inference tools for  longitudinal functional data with a general hierarchical  structure,  where  independent  subjects or units are  observed at repeated times and at each time,
a functional observation (curve) is recorded [see for example, \cite{greven2010},
 \cite{park2015} or \cite{ChenDelicadoMuller2017}  among others].

The goal of  this paper is  to develop  statistical methodology for high-dimensional functional panel data,
which exhibits dependencies in  the time {\bf  and} spatial directions.  In Section \ref{sec2}  we introduce a new concept of
physical dependent processes  in the space of square integrable functions on the interval $[0,1]$.
 Our approach is similar   in spirit to the model of $m$-approximable functional time series
introduced by  \cite{hormkoko2010}, but our formulation further adopts the idea of basis decomposition of functional data and consequently separates the functional index and time index in the mathematical representation of the functional processes. We also refer to the work of  \cite{bosq2002}, \cite{bradley2007} or  \cite{panaretos2013}
 for other  concepts modelling functional dependent data such as  mixing conditions or  autocorrelations and cumulants.
 It is found in our theoretical investigations that it is simpler and more efficient to establish Gaussian approximation, chaining and bootstrap results for high dimensional functional time series using our physical representation. 
 In Section \ref{sec3}
 we provide some probabilistic results  which are  useful for the  statistical 
 analysis of high dimensional functional pane data. In particular we derive a Gaussian approximation for sums (with respect to time) of high dimensional panels, which is used  for the construction  of a multiplier bootstrap procedure
 to approximate  the distribution of sums uniformly with respect to the spatial   dimension.  These results have numerous applications in the statistical inference of high-dimensional functional time series
 and we  illustrate the potential  of our approach in Section \ref{sec4}. Here we derive joint simultaneous confidence bands for the mean functions of a high-dimensional functional time series
  and construct a test to check  the hypothesis of  parallelism of the  mean functions. Section \ref{sec5} is devoted to a small simulation study to illustrate
  the finite sample properties of the new methodology. We also present a data example analyzing yearly temperature curves from from different
   Canadian cities.  Finally,   all proofs are deferred to the online supplemental material of the paper.

\section{Physical representation of functional time series} \label{sec2}
\def\theequation{2.\arabic{equation}}
\setcounter{equation}{0}

In this section we introduce  the basic functional time series model considered in this paper.
To be precise, let  ${\cal L}^2[0,1]$ denote  the set of all real-valued square integrable functions defined on the interval $[0,1]$
and let $\{X_{i}\}_{i \in \mathbb{Z}}$  be   a stationary functional time series  in ${\cal L}^2[0,1]$.
 We begin with a motivating example for the general model defined below.

\begin{example}  \label{ex1}
{\rm  Let  ${\cal B}= \{B_j\}_{j=0}^\infty$  be a  basis of   ${\cal L}^2[0,1]$, then the  $i$-th element  $X_{i}$
of   a time series  $\{X_{i}\}_{i \in \mathbb{Z}}$   in ${\cal L}^2[0,1]$ can be represented   as
\begin{eqnarray}\label{eq:basis}
X_i(u)=\sum_{j=0}^\infty L_{i,j}B_j(u) ~~(i \in \mathbb{Z} )~,
\end{eqnarray}
where   $L_{i,j}=\int^b_a X_i(u)B_j(u)du$ is the $j$th Fourier coefficient of $X_{i}$ with respect to the basis ${\cal B}$.
If ${\bf L}_{i}=(L_{i,0},L_{i,1},\cdots)^\top$ denotes the infinite-dimensional vector of Fourier coefficients of $X_i$, one can use
any   dependence concept for  the infinite dimensional stationary time series $\{{\bf L}_i\}_{i\in{ \mathbb{Z}}}$
to describe  the dependence structure of the functional time series $\{X_{i}\}_{i \in \mathbb{Z}}$ in ${\cal L}^2[0,1]$.
Following   \cite{wu2005nonlinear}, a general model for the infinite dimensional stationary time series $\{{\bf L}_i\}_{i\in{ \mathbb{Z}}}$ is
given by
$$
{\bf L}_i={\bf G}(\FF_i), \quad i\in{ \mathbb{Z}},
$$
where ${\bf G}: {\cal S}^{\N} \to \R^{\N}$ is a given filter, $\FF_i=(\cdots, \eta_{i-1}, \eta_i)$  and $\{\eta_i\}_{i \in \mathbb{Z}}$ is a sequence  of independent and identically distributed
random variables taking values in some space ${\cal S}$. In this case  each element $X_i$  of the  time series can be written in the form $X_i(u) =H(u,\FF_i)$ for an appropriate function
$H$ mapping a sequence of Fourier coefficients to its corresponding function  in ${\cal L}^2[0,1]$.
}
\end{example}

Using  Example \ref{ex1}  as motivation,  we now introduce the following formulation of a stationary functional time series and the associated dependence concept, which will be considered in this paper.

\begin{definition}\label{def:concept}
A stationary functional time series $\{X_i(u)\}_{i\in\mathbb{Z}}$ in ${\cal L}^2[0,1]$ has a physical representation, if there exists a measurable function $H: [0,1]  \times {\cal S}^{\mathbb{N}} \to \R $ such that
\begin{eqnarray}\label{eq:model}
X_i(u) =H(u,\FF_i), \quad  \forall u\in[0,1],
\end{eqnarray}
where $\FF_i=(\cdots, \eta_{i-1}, \eta_i)$ and $\{\eta_i\}_{i \in \mathbb{Z}}$  is a sequence of independent identically
distributed random elements in ${\cal S}$.
\end{definition}

Note that  a functional  time series  $\{X_{i}\}_{i\in{\mathbb{Z}}} \subset {\cal L}^2[0,1]$  of the form (\ref{eq:model}) is strictly stationary by this definition.
In order to specify its dependence, we define for a real valued random variable  $Z$ and a constant $q \geq 1$ (in the case of existence)  its norm  $|| Z ||_{q} = \big ( \mathbb{E} [| Z|^{q}| \big )^{{1/q}}$ and introduce the following measure of physical dependence.

\begin{definition}\label{def:dependence}
Let $\{X_{i}\}_{i\in{\mathbb{Z}}}$   denote  a  functional time series  in ${\cal L}^2[0,1]$ of the form (\ref{eq:model}), such that $\| X_i(u) \|_q < \infty $. The dependence measures
of $\{X_{i}\}_{i\in{\mathbb{Z}}}$  are defined by  
\begin{eqnarray}\label{eq:dep}
\delta_H(k,q)=\sup_{u\in[0,1]}\|H(u,\FF_k)-H(u,\FF^*_k)\|_q, \quad k \ge 0,
\end{eqnarray}
where   $\FF^*_k=(\cdots, \eta_{-2}, \eta_{-1}, \eta^*_0,\eta_1,\cdots,\eta_{k-1}, \eta_k)$ and $\eta_0^*$ is identically distributed as $\eta_0$ and is independent of the sequence $\{\eta_i\}_{i\in{\mathbb Z}}$.
\end{definition}
In (\ref{eq:model}) the representation $H$ can be viewed as a filter or data generating function of a stochastic system and $\{\eta_i\}_{i\in\mathbb{Z}}$ can be viewed as random shocks or innovations of the system. Adopting the latter point of view, the physical dependence measures $\delta_H(k,q)$ quantify the influence of the innovations $k$ steps ahead on the current output of the system. Weak dependence is characterized by the fast decay of $\delta_H(k,q)$ as $k$ grows. Note that our definition of a functional time series and their dependence measures are  different from the  popular concept of ${\cal L}^{p}$-$m$-approximability  (see, for example,  \cite{hormkoko2010} and \cite{HorvathKokoskza2012}), where $X_i(u)$ is expressed as $G(\epsilon_i(u),\epsilon_{i-1}(u),\cdots)$ with $\{\epsilon_i(u)\}_{i\in\mathbb{Z}}$  i.i.d. random functions and the dependence is measured by the changes in $X_i(u)$  when certain $\epsilon_j(u)$ in the past are replaced with i.i.d. copies. Both formulations  are based 
 on Bernoulli shifts and coupling to quantify and control the dependence strength of functional time series. However,  the formulation considered in this paper further adopts the idea of basis decomposition of functional data and hence separates the functional index $u$ and time index $i$ in (\ref{eq:model}). It is found in our theoretical investigations that it is simpler and more efficient to establish Gaussian approximation and bootstrap results for high dimensional functional time series using the representation (\ref{eq:model}), conditions on the decay of the dependence measure (\ref{eq:dep}) and the associated basis decomposition.

We continue giving two examples on how to check the rate of decay  of $\delta_H(k,q)$ for linear and nonlinear functional time series models.

\begin{example}\label{ex2}
{\rm Consider the following MA($\infty$) functional linear model
\begin{eqnarray*} 
X_i(u)=\sum_{j=0}^\infty \int_{0}^1 a_j(u,v)\epsilon_{i-j}(v)\,dv,
\end{eqnarray*}
where $\{a_j \}_{j \geq 0}$ is a sequence of square integrable functions $a_j : [0,1]^2 \to \mathbb{R}$ satisfying $\sum_{j=0}^\infty\sup_{u,v\in[0,1]}| a_j(u,v)|<\infty$, and $\{\epsilon_{i}\}_{i\in\mathbb{Z}}$ is a sequence of i.i.d. random functions in ${\cal L}^2[0,1]$.
Writing $\epsilon_{i}(u)=\sum_{j=0}^\infty \eta_{i,j}B_j(u)$ with $\eta_i=(\eta_{i,j})^\top_{j\ge 0}$, we see that $X_i$ can be represented in the form of (\ref{eq:model}). Assume that, for some $q>1$, $\sup_{u\in[0,1]}\|\epsilon_{i}(u)\|_q<\infty$, then simple calculations using Definition \ref{eq:dep} show that the physical dependence measures in (\ref{eq:dep}) satisfy
$$\delta_H(k,q)=O\Big(\sup_{u\in[0,1]}\int_{0}^1|a_k(u,v)|\,dv\Big).$$}

\begin{example}\label{ex3}
{\rm In this example we continue our investigation of the nonlinear model introduced in Example \ref{ex1}. Recall the representation (\ref{eq:basis}) and let     $b_j=\sup_{u\in[0,1]}|B_j(u)|$ and $l_{j}=\E|L_{i,j}|$. The rate of decay   of $l_j$ is determined by the smoothness of the functions $X_i$ as well as the basis functions $B_j$ used in the decomposition (\ref{eq:basis}). See Proposition \ref{prop:1} below for the calculation of $l_j$ when $X_i(u)$ is twice continuously differentiable and $B_j$ is the cosine basis. Now write
$$X_i(u)=\sum_{j=0}^\infty l_j \tilde L_{i,j}B_j(u),$$
with $\tilde L_{i,j}= L_{ij}/l_j=\tilde G_{j}(\FF_i)$. Denote the physical dependence measure of $\{\tilde L_{i,j}\}_{i\in\mathbb{Z}}$ by (cf.  \cite{wu2005nonlinear}) $\theta_{j}(k,q):=\|\tilde G_{j}(\FF_k)-\tilde G_{j}(\FF^*_k)\|_q$. Then using (\ref{eq:dep}), we obtain that
\begin{eqnarray}\label{bound_ex3}
\delta_H(k,q)=O\Big( \sum_{j=0}^\infty l_jb_j\theta_j(k,q)\Big).
\end{eqnarray}
For a wide class of nonlinear time series models, \cite{wu2005nonlinear} contains detailed calculations of the  rate of decay  of their physical dependence measures. As $l_j$ and $b_j$ can be calculated as discussed above, the estimate (\ref{bound_ex3}) can be used to bound the dependence measures of a wide class of functional time series. A further simplification can be carried out if     $ \sum_{j=0}^\infty l_jb_j<\infty$. In this case we have $\delta_H(k,q)=O(\theta(k,q))$, where $\theta(k,q)=\sup_{j\ge 0}\theta_j(k,q)$.}
\end{example}

\end{example}

\medskip

Throughout this paper we will work with the basis $\{ \cos(k\pi u)\}_{k=0,1,\ldots}$ of ${\cal L}^{2}[0,1]$. If   $\{X_i \}_{i\in{\mathbb{Z} }}$ is a functional time series in   $ {\cal L}^2[0,1]$
 such that  (\ref{eq:model}) holds for some filter $H$, it   can be represented as a Fourier series
  $$
X_i(u)=\sum_{k=0}^\infty a_{i,k}\cos(k\pi u),
$$
where
\begin{equation}\label{fcoeff}
  a_{i,k} = 2 \int^1_0 \cos (k \pi u) X_i (u) du,
\end{equation}
$k\ge 1$ and $a_{i,0}=\int^1_0 X_i (u) du$ are random variables which can be written in the form
$$
a_{i,k}=G_k(\FF_i) = 2  \int^1_0  \cos (k \pi u)  H(u,\FF_i)du
$$
for some filter $G_k$. The following result specifies the rate of decay of the Fourier coefficients under smoothness conditions on the function $X_i$.
\begin{proposition}\label{prop:1}
Suppose that $\{X_i \}_{i \in \mathbb{Z}}$ is a time series in $\mathcal{L}^2[0,1]$ such that  (\ref{eq:model}) is satisfied and $\| X_i(u)\|_q < \infty$ for some $q \geq 2$. Assume $X_i$ is twice continuously differentiable on the interval [0,1] a.s. and that
 \begin{equation}\label{ass0}
  \|X_1(0)\|_q+\|X_1'(0)\|_q+\sup_{u\in[0,1]}\|X_1''(u)\|_q<\infty.
\end{equation}
  Then the Fourier coefficients in  (\ref{fcoeff}) satisfy
$$
\|a_{i,k}\|_q\le {C_q \over k^2}
$$
 for some constant $C_q$ which is independent of $i,k$. Furthermore,
 $$
 \|a_{i,k}-a^*_{i,k}\|_q=O(\min(\delta_H(i,q),1/k^2)),
 $$
  where $a^*_{i,k}=G_k(\FF_i^*) = \int^1_0 \cos(k\pi u)  H (u, \FF^*_i)du$.
 \end{proposition}
\noindent Proposition \ref{prop:1} establishes the rate of decay of $a_{i,k}$ as a function of $k$ as well as the dependence measure of $a_{i,k}$ when viewed as a time series indexed by $i$. In the following we will consider the standardized sum
\begin{equation*} 
 S_n(u)= \frac {1}{\sqrt{n}} \sum_{i=1}^nX_i(u)
\end{equation*}
  and its approximation
\begin{equation*} 
  S_{n}(k,u)=\frac {1}{\sqrt{n}} \sum_{i=1}^nX_i(k,u),
\end{equation*}
 where $X_i(k,u)= \sum_{j=1}^k a_{i,j}\cos(j\pi u)$ is the $k$th Fourier sum of $X_i(u)$.
\begin{proposition}\label{prop:2}
If the assumptions of Proposition \ref{prop:1} are satisfied,  $\E(X_i(u))=0$ and $\delta_H(i,q)=O(i^{-\beta})$ for some $\beta>2$, we have
$$\Big \|\sup_{u\in[0,1]}|S_n(u)-S_{n}(k,u)|\Big \|_q\le Ck^{(2-\beta)/\beta},$$
where $C$ is a finite constant which does not depend on $k$ or $n$.
\end{proposition}

\medskip
Proposition \ref{prop:2} establishes that $S_{n}(u)$ can be well approximated by $S_n(k,u)$ if $k$ is sufficiently large. The latter is an important result due to the fact that the theoretical investigation of $S_n(k,u)$ is much easier as it can be written as a linear combination of {\it finitely many} random variables.

We conclude this section with a result on {chaining of functional time series}. To be precise
define a grid ${\cal U}_n=\{t_{i,n}\}_{i=0}^{l_n}$ where $t_{i,n}=i/l_n$ and $l_n$ is a positive integer that diverges to infinity. The following proposition establishes that $\sup_{u\in[0,1]}|S_{n}(k,u)|$ can be well approximated by $\max_{u\in {\cal U}_n}|S_{n}(k,u)|$ for a sufficiently dense grid (that is   $l_n \to \infty$).
\begin{proposition}\label{prop:3}
If the assumptions of Proposition \ref{prop:2} are satisfied, we have
\begin{eqnarray*}
\Big \|\max_{0\le i \le l_n-1}\sup_{u\in [t_{i,n},t_{i+1,n}]}|S_{n}(k,u)-S_{n}(k,t_{i,n})| \Big \|_q\le C{ k^{2/\beta} \over l^{1-1/q}_n} ,
\end{eqnarray*}
where $C$ is a finite constant which does not depend on $k$ or $n$.
\end{proposition}

\section{Main results} \label{sec3}
\label{sec3}
\def\theequation{3.\arabic{equation}}
\setcounter{equation}{0}

In this section we derive  a Gaussian approximation for partial sums of  high dimensional functional time series which can be  used  to define a multiplier bootstrap  procedure.
Throughout this paper  $\{\bX_i\}_{i=1}^n$ denotes   an $r$-dimensional stationary functional times with  $\bX_i = (X_{i,1} ,  \ldots , X_{i,r})^\top $, where each component
$X_{i,j}$   of the vector
 $\bX_i $  is an element of ${\cal L}^2 [0,1]$ satisfying  (\ref{eq:model}) for some filter $H_{j}$, that is $X_{i,j}(u)=H_j(u, \FF_i)$  ($j=1,2,\cdots, r$). Consequently, the
$r$-dimensional vector can be represented as
\begin{equation}\label{hvec}
 \bX_i(u)={\bf H}(u,\FF_i)  
\end{equation}
$(i=1 ,  \ldots ,n,$ $u \in [0,1] ),$  where the $r$-dimensional filter is defined by ${\bf H} = (H_{1}, \ldots , H_{r} )^{\top}$. We assume
 that   the dimension
   \begin{equation}\label{rn}
 r\asymp n^{\theta_1}
 \end{equation}
  increases at a polynomial rate  with the sample size $n$, where  $\theta_1>0$  is a given constant. We are interested in the probabilistic properties of
  the  sums of high-dimensional vectors
\begin{equation}\label{snvec}
   \bS_n(u)=\frac {1}{\sqrt{n}} \sum_{i=1}^n\bX_i(u)
\end{equation}
as $n,r \to \infty$, and denote by
  \begin{equation}\label{snvecj}
     S_{n,j}(u)=\frac {1}{\sqrt{n}} \sum_{i=1}^nX_{i,j}(u)~~~~~~(j=1,\ldots , r)
  \end{equation}
  the $j$-th component of the vector  $\bS_n = (S_{n,1}, \ldots, S_{n,r})^T$.

\subsection{Gaussian approximation} \label{sec32}

Consider  the component-wise Fourier expansion of the $r$-dimensional function $\bX_i $
$$
\bX_i(u)=\sum_{j=0}^\infty \ba_{i,j}\cos(j\pi u)
$$
and  define the vector
 \begin{equation} \label{avec}
 \bA_{n,j} = ( A_{n,j,1}, \ldots A_{n,j,r})^{\top}
 = \frac{1}{\sqrt{n}}\sum_{i=1}^n \ba_{i,j}.
\end{equation}
   For $n \in \mathbb{N}$ let $\{\bN_{n,k}\}_{k=0}^\infty$ be a sequence of independent  centered $r$-dimensional Gaussian random variables that preserves 
   the covariance structure of $\{\bA_{n,k}\}_{k=0}^\infty$,    i.e. Cov$(\bA_{n,k})=$Cov$(\bN_{n,k})$, define
\begin{eqnarray}\label{eq:norm}
\bS^N_n(u)= \big (S^N_{n,1}(u), \ldots, S^N_{n,r}(u) \big )^{\top} =  \sum_{k=0}^\infty\bN_{n,k}\cos(k\pi u) 
\end{eqnarray}
and denote by $S^N_{n,j}(u)$  the $j$-th entry of the vector $\bS^N_n(u) = (S^N_{n,1}(u), \ldots , S^N_{n,r}(u))^{\top}$. By the proof of Proposition \ref{prop:2}, we have  for $l$-th coordinates  $A_{n,j,l}$ of
the vectors  $\bA_{n,j}$ in  (\ref{avec})
$$
\sum_{j=0}^\infty \|A_{n,j,l}\|_q<\infty ~,
$$  
and therefore  the random variable $\bS^N_n(u)$ in (\ref{eq:norm}) is well defined (almost surely).

\begin{theorem}\label{thm:1}
Let $\{ \bX_i \}_{i \in \mathbb{Z}}$ denote a stationary $r$-dimensional time series satisfying (\ref{hvec}).
For each $j$, suppose that $X_{i,j}$ is twice continuously differentiable on the interval $[0,1]$ a.s.;
  satisfies $\|X_{ij}\|_q < \infty$ and (\ref{ass0})  for some $q\ge 4$. Assume $\E (\bX_i(u))=0$ and
\begin{equation}\label{ass3a}
 \max_{1\le j\le r}\delta_{H_j}(i,q)=O(i^{-\beta})
\end{equation}
for some $\beta>3$. If there exists a positive constant $\delta$ such that
\begin{equation}\label{ass3}
  \min_{1\le j\le r}\inf_{u\in[0,1]}\E[S^2_{n,j}(u)]\ge \delta
\end{equation}
 for sufficiently large $n$ and the exponent $\theta_1$ in (\ref{rn}) satisfies
 \begin{equation}\label{ass4}
 \theta_1<f(\beta, q):={q-2 \over  2[1+\beta/[(q-1)(\beta-2)]] },
 \end{equation}
    we have as $n,r \to \infty$
\begin{eqnarray}\label{eq:gauss_approx}
\sup_{x\in \R}\Big | \p \Big [\max_{1\le j\le r}\sup_{0\le u\le 1}|S_{n,j}(u)|\le x \Big ]-\p  \Big [\max_{1\le j\le r}\sup_{0\le u\le 1}|S^N_{n,j}(u)|\le x  \Big  ]\Big |\rightarrow 0.
\end{eqnarray}

\end{theorem}

\medskip

It is easy to show that, under assumptions of Theorem \ref{thm:1},  the weak convergence
$$
S_{n,j}(u)\Rightarrow \mathcal{N}(0,\sigma^2(j,u))
$$
holds for any $j=1, \ldots , r$ and $u \in [0,1] $, where the symbol $\mathcal{N}(\mu, \sigma^2)$ denotes a normal distribution with mean $\mu$ and variance $\sigma^2$.
Consequently, assumption (\ref{ass3}) is rather  mild and means that, uniformly in $u$ and $j$, $S_{n,j}(u)$ will not converge to a degenerate limit.

\subsection{Bootstrapping panel functional time series}
Theorem \ref{thm:1} provides an approximation (in distribution) of $\max_{1 \leq j \leq r} \sup_{u \in [0,1]} | S_{n,j}(u) |$ by a corresponding expression using normally distributed random variables. The latter distribution can be easily simulated if the dependence structure of the process $\{\bN_{n,k} \}^\infty_{k=0}$ would be known. In order to mimic the dependencies, we propose  a multiplier bootstrap method.
To be precise, define for a block size $m$  the local ($r$-dimensional) mean \begin{equation}\label{Tu}
\bT_{i,m}(u)= \frac{1}{m}\sum_{j=i}^{i+m} \bX_j(u)
\end{equation}
and consider the vector
\begin{eqnarray}\label{phim}
\bPhi_{m}(u)=\sqrt{\frac{m}{n-m}}\sum_{i=1}^{n-m}[\bT_{i,m}(u)-\bS_n(u)/\sqrt{n}]N_i,
\end{eqnarray}
where $N_i$ are i.i.d. one-dimensional standard normal random variables. Let $\Phi_{m,j}(u)$ be the $j$-th entry of $\bPhi_{m}(u)$, then we have the following result.

\begin{theorem}\label{thm:2}
Let $\{ \bX_i \}_{i \in \mathbb{Z}}$ denote a stationary $r$-dimensional time series satisfying (\ref{hvec}).
For each $j$, suppose that $X_{i,j}$ is twice continuously differentiable on the interval $[0,1]$ a.s.
  satisfying $\|X_{ij}\|_q < \infty$ and (\ref{ass0})  for some $q>2$ and assume that
   \begin{eqnarray}\label{eq:2lip}
 \|X_{i,j}''(u)-X''_{i,j}(v)\|_q\le C|u-v|
 \end{eqnarray}
 holds for all $u,v \in [0,1]$  for sufficiently large $n$.
 Further assume that and   (\ref{ass3})  holds and that
\begin{equation}\label{ass6}
  \max_{1\le j\le r}[\delta_{H_j}(i,q)+\delta_{H'_j}(i,q)]=O(i^{-\beta})
\end{equation}
 for some $\beta>2$ and $q>4$, where $H'_j(u,\FF_i)$ is the filter corresponding to the derivative $X_{i,j}'(u)$ of the process of $X_{i,j}(u)=H_j(u,\FF_i)$.
 If the block size $m$ satisfies $m \asymp n^\phi$ with $0<\phi<1$ and the exponent $\theta_1$ in (\ref{rn}) satisfies $\theta_1<\phi'q^2/[2(q+1)]$ where $\phi'=-\max\{(\phi-1)/2,-\phi\}$, then on a sequence of events $E_n$ such that $\p(E_n)\rightarrow 1$ if $r,n \to \infty$, we have
\begin{eqnarray*}
\sup_{x\in \R}\Big |\p \Big [\max_{1\le j\le r}\sup_{0\le u\le 1} |\Phi_{m,j}(u)|\le x~ \Big |\{\bX_i\}_{i=1}^n  \Big ]-\p  \Big [\max_{1\le j\le r}\sup_{0\le u\le 1}|S^N_{n,j}(u)|\le x  \Big ]\Big |\rightarrow 0.
\end{eqnarray*}

\end{theorem}

\section{Some statistical applications}
\label{sec4}
\def\theequation{4.\arabic{equation}}
\setcounter{equation}{0}

The results of Section \ref{sec3} can be used for statistical inference of high-dimensional time series,  if the 
statistical analysis is based on a vector of sums of the form (\ref{snvecj}). 
Typical applications are tests for parametric assumptions on the mean function, two (or more) sample comparisons, the construction of confidence regions for parameters of interest, change point analysis, to name just a few.
Exemplarily, we illustrate in this section two applications, namely the construction of joint simultaneous confidence bands for the mean functions and the development of a test that the mean functions at different spatial locations are parallel.

\subsection{Joint  simultaneous confidence bands (JSCB)}\label{sec:jscb}
let $\mathcal{C}^2[0,1]$ denote the space of twice continuously differentiable functions on the interval $[0,1]$, let
$\{\bX_i\}_{i=1}^n$ be an $r$-dimensional stationary functional time series in $(\mathcal{C}^2[0,1])^{r}$ satisfying (\ref{hvec}) and denote  by $\bg(u):=\E(\bX_i(u))=(g_1(u),\ldots,g_r(u))^\top$ the expectation of $\bX_i$. This subsection is devoted to the construction of joint (asymptotic) simultaneous confidence bands (JSCB) for the components of the vector $\bg$ if the sample size $n$ and dimension $r$ converge to infinity. More precisely, we aim to find (random) functions
$c_1,d_1,\ldots,c_r,d_r$ defined on the interval $[0,1]$, such that $c_j(u) \leq d_j(u)$  for all $u \in [0,1]$ and such that the subset
\begin{equation}\label{creg}
  \mathcal{C}_{n,\alpha} = \Big \{ \bh\in (\mathcal{C}^2[0,1])^r \mid c_j(u) \leq h_j(u) \leq d_j(u) ~\forall u \in [0,1] ~\forall j=1,\ldots,r \Big \}.
\end{equation}
of the set of  $r$-dimensional functions $ \bh=  (h_1,\ldots,h_r)^\top $  on the interval $[0,1]$  satisfies
\begin{eqnarray*}
\lim_{n,r\rightarrow \infty}\p( \bg \in   \mathcal{C}_{n,\alpha}  )=1-\alpha
\end{eqnarray*}
for some pre-specified constant $\alpha \in (0,1)$.
To this end, we denote by
\begin{eqnarray*}
v_j(u)=\sqrt{\mbox{Var}(X_{i,j}(u))}, \quad j=1,2,\cdots, r,
\end{eqnarray*}
the variance of  $X_{i,j}(u)$ and make the following assumption.
\begin{itemize}
\item[(V)]   There exists a   positive constant $\delta  >0 $ such that
 $$
 \inf_{u\in[0,1]}\min_{1\le j\le r}v_j(u)\ge \delta
 $$

\end{itemize}

\begin{proposition}\label{prop:scb}

Let $\{ \bX_i \}_{i \in \mathbb{Z}}$ denote a stationary $r$-dimensional time series satisfying (\ref{hvec}).
For each $j$, suppose that $X_{i,j}$ is twice continuously differentiable on the interval $[0,1]$ a.s.
  satisfies $\|X_{ij}\|_q < \infty$ and (\ref{ass0})  for some $q>2$.  Assume that (\ref{ass3a}),  (\ref{ass3}), (\ref{ass4}) and condition (V) hold. Let $\check{X}_{i,j}(u)=(X_{i,j}(u)-g_j(u))$ and define \begin{equation}\label{h7}
\check{S}_{n,j}(u)= \frac {1}{\sqrt{n}}\sum_{i=1}^n \check{X}_{i,j}(u)  \qquad j=1,2,\cdots, r.
\end{equation}
Then, as $r,n \to \infty$, we have
\begin{eqnarray*} 
\sup_{x\in \R}\Big | \p \Big [\max_{1\le j\le r}\sup_{0\le u\le 1}|\check S_{n,j}(u)/v_j(u)|\le x \Big ]-\p  \Big [\max_{1\le j\le r}\sup_{0\le u\le 1}|S^N_{n,j}(u)/v_j(u)|\le x  \Big  ]\Big |\rightarrow 0,
\end{eqnarray*}
where $S^N_{n,j}(u)$ is the $j$th  component of the vector $\bS^N_n = (S^N_{n,1}(u), \ldots, S^N_{n,r}(u))^{\top}$   defined in (\ref{eq:norm}).
\end{proposition}

Proposition \ref{prop:scb} implies that if one can find a critical value $c^n_{1-\alpha}$ such that
$$  \lim_{r,n \to \infty} \p \Big[\max_{1\le j\le r}\sup_{0\le u\le 1}|S^N_{n,j}(u)/v_j(u)|\le c^n_{1-\alpha}\Big ] = 1-\alpha,$$
then a $100(1-\alpha)$$\%$  JSCB for the vector $\bg = (g_1,\ldots,g_r)^\top$   is obtained by  (\ref{creg})
using the functions
\begin{eqnarray}
\label{cj}
  c_j(u) &=& \frac{S_{n,j}(u) - c^n_{1- \alpha} v_j(u)}{\sqrt{n}}, \qquad j=1,\ldots,r, \\
  d_j(u) &=& \frac{S_{n,j}(u) + c^n_{1- \alpha}  v_j(u)}{\sqrt{n}}, \qquad j=1,\ldots,r,
  \label{dj}
\end{eqnarray}
where $S_{n,j}(u)$ is defined in (\ref{snvecj}).  As the standard deviation   $v_{j}$ and the quantile $c^n_{1- \alpha}$ are not known,
we have to estimate these from the data.  In particular $v_{j}^{2}$ is estimated by
the sample variance  of $\{X_{i,j}(u)\}_{i=1}^n$ defined as
\begin{eqnarray*}
\hat{v}^2_j(u)= \frac {1}{n}\sum_{i=1}^n\big(X_{i,j}(u)-S_{n,j}(u)/\sqrt{n}\big)^2.
\end{eqnarray*}
The following lemma establishes that $\hat{v}^2_j$ is a uniformly consistent estimator.
\begin{lemma}\label{lem:hatv}
For each $j$, suppose $X_{i,j}$ is twice continuously differentiable on the interval [0,1] a.s.
and satisfies (\ref{ass0})  for some $q>2$.  Assume that (\ref{ass3a}) holds and that $\theta_1<q/4$. Then we have $r,n \to \infty$
\begin{eqnarray*}
\Big \|\max_{1\le j\le r}\sup_{0\le u\le 1}|\hat{v}^2_j(u)-v^2_j(u)| \Big \|_{q/2}=O(r^{2/q}/\sqrt{n})=o(1).
\end{eqnarray*}
\end{lemma}
In order to obtain the critical value $c^n_{1-\alpha}$, we now utilize the multiplier bootstrap procedure established in Theorem \ref{thm:2}. Detailed steps for the implementation are listed as follows.
\begin{itemize}
\item[(a)] Select block size $m$ such that $m \to \infty$, $m=o(n)$. 
\item[(b)] For a large integer $B$ generate independent standard normal distributed   random variables $\{N_i^j\}_{i=1}^{n-m}$, $j=1,2,\cdots, B$. For each $j$, calculate
\begin{eqnarray*}
\bPhi^j_{m}(u)=\sqrt{\frac{m}{n-m}}\sum_{i=1}^{n-m}[\bT_{i,m}(u)-\bS_n(u)/\sqrt{n}]N_i^j
\end{eqnarray*}
and 
$$
\tilde{\Phi}^j_{m}=\max_{1\le k\le r}\sup_{0\le u\le 1}[|\Phi^j_{m,k}(u)|/\hat v_j(u)],
$$
 where $\Phi^j_{m,k}(u)$ is the $k$-th component of the vector $\bPhi^j_{m}(u) = (\Phi^j_{m,1}(u) , \ldots , \Phi^j_{m,r}(u) )^{\top}$.
\item[(c)] Let $\tilde{\Phi}^{(1)}_{m}\le \cdots\le \tilde{\Phi}^{(B)}_{m}$ be the ordered statistics of $\{\tilde{\Phi}^{j}_{m}\}_{j=1}^B$. Then the quantile $c^n_{1-\alpha}$ is estimated by   $\tilde{\Phi}^{(\lf(1-\alpha)B\rf)}_m$.
\item[(d)]
The JSCB is then defined by (\ref{creg}) using the functions  in (\ref{cj}) and (\ref{dj})  with these estimates, that is
\begin{equation} \label{hd21}
\begin{split}
  c_j(u)= \frac{S_{n,j}(u) - \tilde{\Phi}^{(\lf(1-\alpha)B\rf)}_m \hat v_j(u)}{\sqrt{n}}, \\
  d_j(u) = \frac{S_{n,j}(u) + \tilde{\Phi}^{(\lf(1-\alpha)B\rf)}_m  \hat v_j(u)}{\sqrt{n}}.
\end{split}
   \qquad j=1,\ldots,r,
\end{equation}
\end{itemize}

A data driven rule for the choice of the block length in step (a) will be given in Section \ref{sec:tps}. The following proposition establishes that these definitions  yield an asymptotically correct JSCB.
\begin{proposition}\label{boots_jscb}
Under the conditions of Theorem \ref{thm:2} and Proposition \ref{prop:scb}, we have for the JSCB defined in (\ref{creg}) with the functions $c_j$ and $d_j$ defined in (\ref{hd21})
$$
\lim_{n,r \to \infty}  \lim_{B \to \infty}\mathbb{P} (\bg \in \mathcal{C}_{n,\alpha}) = 1 - \alpha.
$$
\end{proposition}
Proposition \ref{boots_jscb} follows directly from Theorem \ref{thm:2}, Proposition \ref{prop:scb} and Lemma \ref{lem:hatv}. The details are omitted for the sake of brevity. The finite coverage probabilities of these joint simultaneous confidence bands are investigated in Section \ref{sec51} by means of a simulation study.

\subsection{Test of parallelism}\label{sec:tp}
Suppose we observe a panel of functional time series $\{\bX_i \}_{i=1}^n$ of dimension $r$   satisfying (\ref{hvec}).
Adopting the same notation as in Section \ref{sec:jscb}, for $j=1,2,\cdots, r$, let $g_j(u)=\E[X_{i,j}(u)]$ be the mean function of panel $j$. Denote by $g^o_j=\int_{0}^1g_j(u)\,du$ the average of $g_j(u)$ with respect to $u$, $j=1,2,\cdots,r$. In this subsection, we are interested in testing the null hypothesis that the mean functions in the different panels are parallel, that is
\begin{equation}\label{h0}
{H}_0 : g_1(\cdot)\parallel g_2(\cdot)\parallel\cdots\parallel g_r(\cdot),
\end{equation}
versus the alternative that at least two mean functions are not parallel. Here the symbol $\parallel$ denotes parallelism of two functions; that is, $g_i \parallel g_j $ if and only if the difference $g_i(u)-g_j(u)$ is a constant function.

Now let
$$
\hat{g}_j(u)=\frac{1}{n} \sum_{i=1}^nX_{i,j}(u)  \mbox{ and }  \hat{g}^o_j=\int_{0}^1\hat{g}_j(u)\,du
$$
 be the corresponding sample estimates of $g_j(u)$ and $g_j^o$, respectively, and define  
 $$
 \nu_j(u)=g_j(u)-g^o_j ~;~~~ \hat \nu_j(u)=\hat{g}_j(u)-\hat g^o_j.
 $$
  To test the hypothesis  (\ref{h0}) we    consider the following  statistic

\begin{eqnarray}\label{tstat}
T_n=\sqrt{n}\max_{1\le j<k\le r}\sup_{0\le u\le 1}|\hat \nu_j(u)- \hat \nu_k(u)|/v_{j,k}(u),
\end{eqnarray}
where $v_{j,k}(u)>0$ is a measure of scale for $\hat \nu_j(u)- \hat \nu_k(u)$ which will be defined below. Observe that the null hypothesis in (\ref{h0}) is satisfied if and only if $ \nu_j(u)=\nu_k(u)$ for all $u\in[0,1]$ and all pairs $k$ and $j$, $1\le j< k\le r$. As a result, the test statistic $T_n$ should be large if there exists at least one pair of non-parallel functions. Furthermore, $T_n$ can be viewed as a family-wise-error-controlled pair-wise test of parallelism among the curves $g_1,\ldots,g_r$ and one can use the test $T_n$ to decide which curves are parallel to each other and which ones are not with a given family-wise error   (asymptotically). Consequently, the statistic $T_n$ can be used to cluster curves with similar shapes together.

In order to obtain the critical values for a test, which rejects the null hypothesis (\ref{h0}) for large values of $T_n$, we define the quantities
\begin{eqnarray*}
W_{i,j}(u)&=&X_{i,j}(u)-\int_{0}^1X_{i,j}(u)\,du \\
W_{i,j,k}(u)&=&W_{i,j}(u)-W_{i,k}(u)
\end{eqnarray*}
and
\begin{equation}\label{h9}
 \mathring S_{n,j,k}(u)=\frac{1}{\sqrt{n}}\sum_{i=1}^nW_{i,j,k}(u), \quad  1\le j < k\le r.
\end{equation}
With these notations it is easy to see that the statistic $T_n$ in (\ref{tstat}) can be represented as
\begin{equation}\label{h10}
T_n=\max_{1\le j<k\le r}\sup_{0\le u\le 1}|\mathring S_{n,j,k}(u)|/v_{j,k}(u).
\end{equation}
From this representation we observe immediately that Theorems \ref{thm:1} and \ref{thm:2} can be used to determine its critical values. Furthermore, it is easy to see from the above representation that a natural choice for the scaling factors is
$$v_{j,k}(u)=\sqrt{\mbox{Var}[W_{i,j,k}(u)]},$$
where in practice, one replaces $v_{j,k}(u)$ with its sample version, that is
$$
\hat v_{j,k}(u)=\Big\{\frac{1}{n}\sum_{i=1}^n\Big[W_{i,j,k}(u)-\frac{1}{n}\sum_{s=1}^nW_{s,j,k}(u)\Big]^2\Big\}^{1/2}.
$$

Let   $\mathring\bW_i(u)$ be the vector of dimension $r(r-1)/2$  with entries $W_{i,j,k}(u)$, $1\le j < k\le r$, consider its (coordinate-wise) Fourier expansion
\begin{equation*} 
\mathring\bW_i(u) =\sum_{j=0}^\infty \mathring\ba_{i,j}\cos(j\pi u)
\end{equation*}
and define the $r(r-1)/2$-dimensional vector
$$
 \mathring \bA_{n,j}=\frac{1}{\sqrt{n}}\sum_{i=1}^n \mathring \ba_{i,j} ~.
 $$
For $n \in \mathbb{N}$ let $\{\bN_{n,k}\}_{k=0}^\infty$ be a sequence of independent 
centered $r(r-1)/2$-dimensional Gaussian vectors that preserves the covariance structure of $\{\mathring\bA_{n,k}\}_{k=0}^\infty$, i.e. 
$$\mbox{Cov}
(   \{\mathring \bA_{n,k}\}_{k=0}^\infty)=\mbox{Cov}(\{\bN_{n,k}\}_{k=0}^\infty),$$
 and define the
vector
\begin{eqnarray*} \label{eq:norm_tp}
\mathring\bS^N_n(u)=
\sum_{l=0}^\infty\mathring\bN_{n,l}\cos(l\pi u) = \big ( \mathring S^N_{n,j,k}(u)\big)_{1\le j < k\le r},
\end{eqnarray*}
where the random variables  $\mathring S^N_{n,j,k}(u)$ denote the entries of $\mathring \bS^N_n(u)$. By the proof of Proposition \ref{prop:2}, we have $\sum_{k=0}^\infty \|\mathring A_{n,k,a,b}\|_q<\infty$ where
 the random variables   $\mathring A_{n,k,a,b}$
denote the entries  of the vector   $\mathring \bA_{n,k}$  ($1 \leq a < b \leq r$). Hence the random variable $\mathring\bS^N_n(u)$ is well defined almost surely.

\begin{proposition}\label{prop:tp}
For each $j$, suppose that $X_{i,j}$ is twice continuously differentiable on the interval [0,1] a.s.
and satisfies (\ref{ass0})  for some $q>2$.  Assume that (\ref{ass3a}) holds; that (\ref{ass3}) holds with $S^2_{n,j}(u)$ therein replaced by $\mathring S^2_{n,j,k}$; that (\ref{ass4}) holds with $\theta_1$ therein replaced by $2\theta_1$; and that condition (V) hold with $v_j$ therein replaced by $v_{j,k}$, $1\le j<k\le r$. If $n,r \to \infty$, we have, under the null hypothesis (\ref{h0}) of parallelism
\begin{eqnarray*} 
\sup_{x\in \R}\Big | \p \Big [\max_{1\le j<k\le r}\sup_{0\le u\le 1}\Big|\frac{\mathring S_{n,j,k}(u)}{v_{j,k}(u)}\Big|\le x \Big ]-\p  \Big [\max_{1\le j<k\le r}\sup_{0\le u\le 1}\Big|\frac{\mathring S^N_{n,j,k}(u)}{v_{j,k}(u)}\Big|\le x  \Big  ]\Big |\rightarrow 0.
\end{eqnarray*}
\end{proposition}
Proposition \ref{prop:tp} reveals that under the null hypothesis of parallel mean functions $g_1,\ldots,g_r$  the distribution of the statistic $T_n$ in (\ref{tstat}) can be well approximated by the law of the $L^\infty$ norm of a high-dimensional vector of Gaussian random functions. Theorem \ref{thm:2} implies that the multiplier bootstrap can be used to approximate this $L^\infty$ norm. A combination of these results motivates the following bootstrap test for the hypothesis (\ref{h0}).
\begin{itemize}
\item[(i)] Select a block size $m$, such that $m \to \infty$, $m=o(n)$.
\item[(ii)] For a large integer $B$ generate independent standard normal distributed random variables $\{N_i^h\}_{i=1}^{n-m}$, $h=1,2,\cdots, B$. For each $h$ and pair $(j,k)$, $1\le j<k\le r$, calculate
\begin{eqnarray*}
\mathring \Phi^h_{m,j,k}(u)=\sqrt{\frac{m}{n-m}}\sum_{i=1}^{n-m}[\mathring T_{i,m,j,k}(u)-\mathring S_{n,j,k}(u)/\sqrt{n}]N_i^h/\hat{v}_{j,k}(u),
\end{eqnarray*}
where $\mathring T_{i,m,j,k}(u)=\sum_{a=i}^{i+m} W_{a,j,k}(u)/m$, and define
 $$
\breve{\Phi}^h_{m}=\max_{1\le j<k\le r}\sup_{0\le u\le 1}|\mathring \Phi^h_{m,j,k}(u)|.
$$
\item[(iii)] Let $\breve{\Phi}^{(1)}_{m}\le \cdots\le \breve{\Phi}^{(B)}_{m}$ be the ordered statistics of $\{\breve{\Phi}^{j}_{m}\}_{j=1}^B$ and estimate the quantile $\mathring c_{1-\alpha,n}$ of the distribution of $T_n$  in (\ref{h10}) by  $\breve{\Phi}^{(\lf(1-\alpha)B\rf)}_m$.
    \item[(iv)] Reject the null hypothesis (\ref{h0}) of parallel mean functions, whenever
    \begin{equation}\label{boottest}
      T_n >  \breve{\Phi}^{(\lf(1-\alpha)B\rf)}_m .
    \end{equation}
\end{itemize}

A data driven rule for the choice of the block size in step (a) will be given in Section \ref{sec:tps}.
The asymptotic validity of the above multiplier bootstrap procedure is established in the following two propositions.

\begin{proposition}\label{boots_tp}
Suppose that (\ref{eq:2lip}) and (\ref{ass6})   hold and that $\theta_1<\phi'q^2/[4(q+1)]$. Further assume that the conditions 
of  Proposition \ref{prop:tp} are satisfied. Then, under the null hypothesis (\ref{h0}) of parallel mean functions we have

\begin{eqnarray*}
\lim_{n,r \to \infty}  \lim_{B \to \infty}\p \big (T_n >   \breve{\Phi}^{(\lf(1-\alpha)B\rf)}_m  \big ) = \alpha .
\end{eqnarray*}

\end{proposition}
Proposition \ref{boots_tp} follows directly from Theorem \ref{thm:2}, Proposition \ref{prop:tp} and Lemma \ref{lem:hatv}. The details are omitted for the sake of brevity. Next, we study the power  of the test  considering the local alternative
\begin{itemize}
\item[${\bf H}_a$ :] There exists a pair $(j,k)$, $1\le j<k\le r$, such that
\begin{equation}\label{halt}
  \inf_{c\in\R}\sup_{u\in[0,1]}|g_j(u)-g_k(u)-c|\sqrt{n/\log n}\rightarrow \infty.
  \end{equation}
\end{itemize}
The following proposition shows that the test (\ref{boottest}) achieves asymptotic power 1 under alternatives of the form (\ref{halt}). Hence this test is able to detect alternatives that deviate from the null hypothesis of parallel mean functions at the rate $\sqrt{\log n/n}$.

\begin{proposition}\label{power_tp}
If the conditions of Proposition \ref{boots_tp} are satisfied, we have    under alternatives of the form (\ref{halt}) that
\begin{eqnarray*}
\lim_{n,r \to \infty}  \lim_{B \to \infty} \p \big (T_n > \breve{\Phi}^{(\lf(1-\alpha)B\rf)}_m  \big )= 1.
\end{eqnarray*}
\end{proposition}

\subsection{Tuning parameter selection}\label{sec:tps}
To implement the multiplier bootstrap, one needs to choose the block size $m$. Observe that the quality of the bootstrap depends on how well the conditional covariance operator of $\bPhi_{m}$ in (\ref{phim})  approximates the covariance operator of $\bS^N_n$ in (\ref{eq:norm}). The tuning parameter $m$ controls the accuracy for the latter approximation. When $m$ is too large, the variance of the conditional covariance operator of $\bPhi_{m}$ will be too large; while if $m$ is too small, then the latter conditional covariance operator will have too much bias. In this paper, we adopt the minimum volatility (MV) method proposed in \cite{politis1999subsampling} to select the block size $m$. The idea behind the MV method is that the conditional covariance operator of $\bPhi_{m}$ should behave stably as a function of $m$ when $m$ is in an appropriate range. Therefore one could select $m$ that minimizes the variability of the conditional covariance operator of $\bPhi_{m}$ as a function of $m$.  We refer the readers to \cite{politis1999subsampling} and \cite{zhou2013} for more detailed discussions of the MV method and its applications in time series analysis. 

Specifically, one  first determines a sequence of equally-spaced candidate block sizes $m_1<m_2<\cdots<m_k$ and defines $m_0=2m_1-m_2$ and $m_{k+1}=2m_k-m_{k-1}$. For each $i \in \{0,\ldots,k+1 \}$, one calculates
\begin{eqnarray*}
\Xi^{(i)}_{j}(u):=\sqrt{\frac{m_i}{n-m_i}}\sum_{k=1}^{n-m_i}[T_{i,j,m}-S_{n,j}(u)/\sqrt{n}]^2, ~j=1,2,\cdots, r,
\end{eqnarray*}
where $T_{i,j,m}$ denotes the $j$th component of $\bT_{i,m}$. Note that $\Xi^{(i)}_{j}(u)$ is an estimator of the marginal variance operator of $S^N_{n,j}(u)$. Next, one calculates for $i=1,2,\cdots, k$
$$\Xi^{(i)}:=\sum_{j=1}^r\int_{0}^1\mbox{sd}(\{\Xi^{(l)}_{j}(u)\}_{l=i-1}^{i+1})\,du,$$
where $\mbox{sd}$ denotes standard deviation. Observe that $\Xi^{(i)}$ measures the overall variability of $\Xi^{(i)}_{j}(u)$ (with respect to $i$) at candidate block size $m_i$. We finally recommend selecting the block size $m_i$ such that  $i=\argmin_{1\le i\le k}\Xi^{(i)}$. The MV method performs well in our  simulation studies, which are presented in the following section.

\section{Numerical experiments}
\label{sec5}
\def\theequation{5.\arabic{equation}}
\setcounter{equation}{0}

In this section we illustrate the application of the methodology by means of a small simulation study and by the analysis of a real data example.
\subsection{Simulation study}
\label{sec51}
In the Monte Carlo simulations, we consider the following two types of simple panel time series models:
\begin{itemize}
\item[] The  {\bf PAR($a$)-model}   is defined as the  panel autoregressive functional  model
$$
\bX_i(u)=\bg(u)+ A_r\be_i(u),
$$
where $A_r$ is an $r\times r$ tridiagonal matrix with 1's on the diagonal and $1/2$'s on the off-diagonal,
 $\be_i(u)=(e_{i,1}(u),\cdots,e_{i,r}(u))^\top$ is an $r$-dimensional vector
with independent entries  satisfying the AR-equation
\begin{equation} \label{ar}
e_{i,j}(u)=a e_{i-1,j}(u)+\epsilon_{i,j}(u)~, ~j=1,\ldots ,r~.
\end{equation}
and  $a\in[0,1)$ denotes  the auto-regressive (AR) coefficient that controls the strength of the temporal dependence.
The innovations  in (\ref{ar}) are  given by    i.i.d. random functions
  $$
  \epsilon_{i,j}(u)=\sum_{k=1}^\infty k^{-3}[\cos(2\pi k u)\epsilon_{i,j,k}+\sin(2\pi k u)\epsilon_{i,j,k}] ,~i=1,2,\ldots , ~j=1,\ldots ,r~,
  $$
  where the infinite-dimensional  vector
   $\epsilonv_{i,j}:=(\epsilon_{i,j,1},\epsilon_{i,j,2},\cdots)^\top$ satisfies $\epsilonv_{i,j}=A_\infty \epsilonv'_{i,j} $,
   $A_\infty$  is the infinite-dimensional tridiagonal matrix with $1$'s on the diagonal and $1/2$'s on the off-diagonal and
    \begin{equation} \label{eps}
    \epsilonv'_{i,j}:=(\epsilon'_{i,j,1},\epsilon'_{i,j,2},\cdots)^\top
    \end{equation}
     are i.i.d.  infinite-dimensional random vectors with i.i.d. entries at each component.

\item[] The {\bf PMA($a$)-model} is defined as the   panel moving average  functional  model.
The setup is the same as that of PAR-model except that equation  (\ref{ar}) is replaced by a moving average (MA)
representation 
\begin{equation*} 
 e_{i,j}(u)=\epsilon_{i,j}(u)+a \epsilon_{i-1,j}(u),
 \end{equation*}
  where $a$ is the MA-coefficient that controls the temporal dependence.
\end{itemize}

Observe that the functions  $\bX_i$  in the PAR($a$)- and PMA($a$)-model are twice  but not three times differentiable and
that the components  $ X_{i,j_1}(u)$ and $ X_{i,j_2}(u)$  are independent if $|j_1-j_2|>  1$.
 By construction, the Fourier coefficients $\epsilon_{i,j,k}$ as a sequence of $k$ are correlated.  In the simulation studies, the entries
  $\epsilon'_{i,j,k}$ in the vector (\ref{eps}) are either  i.i.d.  standard normal distributed random variables (denoted by  ${\cal N} (0,1)$)
   or i.i.d.  standardized $t$-distributed random variables with $6$ degrees of freedom (denoted by  $\sqrt{2/3}t_{6}$) to represent light tailed and heavy tailed cases,
   respectively (note that the standard deviation of a $\sqrt{2/3}t_{6}$-distributed  random variable is $1$). Throughout the simulations, the block size is selected by the MV method described in Section \ref{sec:tps}.

\begin{table}[t]
\begin{center}
\begin{tabular}{|l|l|l|l|l|l|l|l|l|l|l|}
\hline
\multicolumn{11}{|c|}{$n=200$}                                                                               \\ \hline
       & \multicolumn{5}{c|}{$1-\alpha=0.95$}               & \multicolumn{5}{c|}{$1-\alpha=0.9$}                \\ \hline
      $r$ & {\scriptsize PAR(0)} & {\scriptsize PAR(.2)} &  {\scriptsize PAR(.5)} & {\scriptsize PMA(.5)} & {\scriptsize PMA(1)}  & {\scriptsize PAR(0)} & {\scriptsize PAR(.2)} &  {\scriptsize PAR(.5)} & {\scriptsize PMA(.5)} & {\scriptsize PMA(1)} \\ \hline
$5$  & .941& .944& .907& .939& .942& .893& .878& .830& .878& .875   \\ \hline
$10$ & .940& .934& .916& .934& .941& .891& .870& .824& .881& .878   \\ \hline
$20$ & .938& .943& .919& .945& .936& .884& .890& .817& .887& .869   \\ \hline
$40$ & .942& .942& .914& .943& .941& .885& .871& .818& .873& .873   \\ \hline
$80$ & .943& .941& .912& .944& .934& .902& .889& .808& .871& .867   \\ \hline
\multicolumn{11}{|c|}{$n=400$}                                                                               \\ \hline
       & \multicolumn{5}{c|}{$1- \alpha=0.95$}               & \multicolumn{5}{c|}{$1- \alpha=0.9$}                \\ \hline
       $r$ & {\scriptsize PAR(0)} & {\scriptsize PAR(.2)} &  {\scriptsize PAR(.5)} & {\scriptsize PMA(.5)} & {\scriptsize PMA(1)} & {\scriptsize PAR(0)} & {\scriptsize PAR(.2)} &  {\scriptsize PAR(.5)} & {\scriptsize PMA(.5)} & {\scriptsize PMA(1)} \\ \hline
$5$  & .948& .947& .941& .945& .944& .892& .880& .887& .899& .893   \\ \hline
$10$ & .948& .957& .934& .939& .939& .898& .890& .862& .877& .872   \\ \hline
$20$ & .954& .953& .946& .943& .943& .897& .891& .878& .885& .891   \\ \hline
$40$ & .955& .944& .947& .949& .940& .905& .885& .873& .885& .888   \\ \hline
$80$ & .956& .946& .942& .955& .943& .903& .908& .857& .896& .880    \\ \hline
\end{tabular}
\end{center}
\caption{{\it Simulated coverage probabilities of the JSCB   in  the  PAR- and PMA-model. The errors in (\ref{eps}) are   $\mathcal{N}(0,1)$-distributed and $1-\alpha$ represents the nominal coverage probability.}}\label{tab:1}
\end{table}

\subsubsection{Coverage accuracy of JSCB} \label{sec511}
In order to investigate the approximation of the confidence level of the JSCB for finite sample sizes, we consider the function $\bg(u)=0$ in the  PAR($a$)- and PMA($a$)-model. The sample size $n$ is chosen as $200$ and $400$ and the dimension $r$ varies from $5$ to $80$ in order  to investigate the impact of the dimensionality on the coverage probability. The number of bootstrap replications   is chosen as $B=1000$ and for each scenario $1000$ simulation runs are performed. The simulated coverage probabilities of the JSCB under both light and heavy tailed errors and various levels of time series dependence are reported in Tables \ref{tab:1} and \ref{tab:2}.

We observe that the performances of the JSCB under light and heavy tailed errors are similar. When $n=200$, the JSCB is reasonably accurate for all dimensions when the time series dependence is moderate. 
Under stronger time series dependence, such as for the PAR($0.5$)-model, the coverage probabilities of the JSCB are slightly smaller than the nominal confidence level. However, the approximation improves significantly if the sample size  increases to $n=400$. In this case, the coverage of the JSCB is reasonably accurate in all cases under consideration.  The impact of the dimension on the accuracy of the approximation of the confidence level is hardly visible, but we observe a slightly better performance of the JSCB for smaller dimensions $r=5,10$ compared to the cases $r=20,40,80$.

\begin{table}[t]
\begin{center}
\begin{tabular}{|l|l|l|l|l|l|l|l|l|l|l|}
\hline
\multicolumn{11}{|c|}{$n=200$}                                                                               \\ \hline
       & \multicolumn{5}{c|}{$1-\alpha=0.95$}               & \multicolumn{5}{c|}{$1-\alpha=0.9$}                \\ \hline
     $r$  & {\scriptsize PAR(0)} & {\scriptsize PAR(.2)} &  {\scriptsize PAR(.5)} & {\scriptsize PMA(.5)} & {\scriptsize PMA(1)}  & {\scriptsize PAR(0)} & {\scriptsize PAR(.2)} &  {\scriptsize PAR(.5)} & {\scriptsize PMA(.5)} & {\scriptsize PMA(1)} \\ \hline
$5$  & .941& .944& .919& .939& .937& .884& .887& .83 & .873& .874   \\ \hline
$10$ & .942& .943& .918& .938& .943& .895& .880& .827& .873& .878   \\ \hline
$20$ & .942& .944& .911& .936& .945& .892& .880& .803& .876& .883   \\ \hline
$40$ & .934& .943& .916& .937& .944& .877& .888& .812& .877& .885   \\ \hline
$80$ & .944& .954& .925& .935& .940& .877& .871& .814& .872& .868   \\ \hline
\multicolumn{11}{|c|}{$n=400$}                                                                               \\ \hline
       & \multicolumn{5}{c|}{$1-\alpha=0.95$}               & \multicolumn{5}{c|}{$1-\alpha=0.9$}                \\ \hline
      $r$  & {\scriptsize PAR(0)} & {\scriptsize PAR(.2)} &  {\scriptsize PAR(.5)} & {\scriptsize PMA(.5)} & {\scriptsize PMA(1)} & {\scriptsize PAR(0)} & {\scriptsize PAR(.2)} &  {\scriptsize PAR(.5)} & {\scriptsize PMA(.5)} & {\scriptsize PMA(1)} \\ \hline
$5$  & .952& .939& .939& .948& .940& .906& .889& .883& .885& .888   \\ \hline
$10$ & .947& .945& .936& .944& .943& .898& .899& .875& .883& .874   \\ \hline
$20$ & .955& .952& .923& .946& .950& .904& .886& .857& .881& .890   \\ \hline
$40$ & .946& .944& .941& .950& .942& .896& .887& .861& .887& .885   \\ \hline
$80$ & .947& .956& .929& .954& .945& .893& .901& .856& .902& .886    \\ \hline
\end{tabular}
\end{center}
\caption{{\it Simulated coverage probabilities of
 the JSCB   in the PAR- and PMA-model. The errors in (\ref{eps}) are    $\sqrt{2/3}t_6$-distributed and $1-\alpha$ represents the nominal coverage probability.}}\label{tab:2}
\end{table}

\begin{table}[t]
\begin{center}
\begin{tabular}{|l|l|l|l|l|l|l|l|l|l|l|}
\hline
\multicolumn{11}{|c|}{$n=200$}                                                                               \\ \hline
       & \multicolumn{5}{c|}{$\alpha=0.05$}               & \multicolumn{5}{c|}{$\alpha=0.1$}                \\ \hline
   $r$    & {\scriptsize PAR(0)} & {\scriptsize PAR(.2)} &  {\scriptsize PAR(.5)} & {\scriptsize PMA(.5)} & {\scriptsize PMA(1)}  & {\scriptsize PAR(0)} & {\scriptsize PAR(.2)} &  {\scriptsize PAR(.5)} & {\scriptsize PMA(.5)} & {\scriptsize PMA(1)} \\ \hline
$5$  & .042& .062& .076& .052& .066& .114& .130& .142& .102& .120   \\ \hline
$10$ & .058& .054& .078& .062& .060& .124& .118& .164& .110& .110   \\ \hline
$20$ & .060& .050& .080& .066& .058& .126& .122& .202& .132& .152   \\ \hline
$30$ & .066& .064& .094& .056& .066& .128& .140& .190& .134& .140   \\ \hline
\multicolumn{11}{|c|}{$n=400$}                                                                               \\ \hline
       & \multicolumn{5}{c|}{$\alpha=0.05$}               & \multicolumn{5}{c|}{$\alpha=0.1$}                \\ \hline
   $r$     & {\scriptsize PAR(0)} & {\scriptsize PAR(.2)} &  {\scriptsize PAR(.5)} & {\scriptsize PMA(.5)} & {\scriptsize PMA(1)} & {\scriptsize PAR(0)} & {\scriptsize PAR(.2)} &  {\scriptsize PAR(.5)} & {\scriptsize PMA(.5)} & {\scriptsize PMA(1)} \\ \hline
$5$  & .052& .048& .056& .052& .052& .108& .100& .124& .106& .110   \\ \hline
$10$ & .048& .052& .066& .052& .056& .100& .122& .134& .116& .102   \\ \hline
$20$ & .058& .058& .068& .054& .048& .114& .110& .138& .120& .114   \\ \hline
$30$ & .050& .048& .072& .038& .060& .102& .110& .124& .106& .136   \\ \hline
\end{tabular}
\end{center}
\caption{{\it Simulated type I errors  of the test \eqref{boottest} for parallelism  of the mean functions  in the    PAR- and PMA-model. The errors in (\ref{eps}) are  $\mathcal{N}(0,1)$-distributed and  $\alpha$ represents the nominal level of the test.}}\label{tab:3}
\end{table}

\subsubsection{Accuracy and power of parallelism test}
In order to investigate the finite sample accuracy of the  test for parallelism, we consider the functions
 $$
 g_i(u)=u^2-u+c_i, \quad i=1,2,\cdots, r,
  $$
  in the PAR($a$) and PMA($a$)-model, where $c_i$ are i.i.d. standard normal distributed random variables. Observe that the mean functions $g_i(u)$, $1\le i\le r$ are parallel. Various values for the dimension   $r$ from $5$ to $30$ and various values of the parameter $a$ are selected to investigate the impact of the  dimension and temporal dependence on the accuracy of the test, respectively. Since the parallelism test compares mean curves for every pair of the $r$ component functional time series,  the effective dimensionality of the test is ${r \choose 2}$. Hence when $r=30$, we are effectively testing a 435 dimensional function. Throughout this subsection, the number of bootstrap replications   is chosen as $B=1000$ and for each scenario, $500$ simulations are used for the calculation of the rejection probabilities.

     The simulated Type I error rates of the  test for parallelism under light and heavy tailed errors are reported in Tables \ref{tab:3} and \ref{tab:4}, respectively. We observe    similar results as in Section \ref{sec511}. There are no significant differences in the   approximation of the nominal level   for light and   heavy tails, and the test \eqref{boottest} is reasonably accurate when the time series dependence is not too strong. Moreover, the 
      test becomes  also  reasonably accurate under stronger temporal dependence for larger sample size (here $n= 400$). The 
      approximation of the nominal level by  the test \eqref{boottest} is only slightly better for lower dimensions ($r=5,10$) compared with   higher dimensions ($r=20,30$).

\begin{table}[t]
\begin{center}
\begin{tabular}{|l|l|l|l|l|l|l|l|l|l|l|}
\hline
\multicolumn{11}{|c|}{$n=200$}                                                                               \\ \hline
       & \multicolumn{5}{c|}{$\alpha=0.05$}               & \multicolumn{5}{c|}{$\alpha=0.1$}                \\ \hline
    $r$   & {\scriptsize PAR(0)} & {\scriptsize PAR(.2)} &  {\scriptsize PAR(.5)} & {\scriptsize PMA(.5)} & {\scriptsize PMA(1)}  & {\scriptsize PAR(0)} & {\scriptsize PAR(.2)} &  {\scriptsize PAR(.5)} & {\scriptsize PMA(.5)} & {\scriptsize PMA(1)} \\ \hline
$5$  & .050& .054& .070& .050& .058& .098& .118& .166& .116& .122   \\ \hline
$10$ & .050& .056& .086& .058& .054& .104& .124& .172& .128& .134   \\ \hline
$20$ & .048& .058& .088& .054& .066& .108& .108& .196& .116& .140   \\ \hline
$30$ & .064& .054& .102& .066& .042& .112& .118& .222& .134& .124   \\ \hline
\multicolumn{11}{|c|}{$n=400$}                                                                               \\ \hline
       & \multicolumn{5}{c|}{$\alpha=0.05$}               & \multicolumn{5}{c|}{$\alpha=0.1$}                \\ \hline
    $r$    & {\scriptsize PAR(0)} & {\scriptsize PAR(.2)} &  {\scriptsize PAR(.5)} & {\scriptsize PMA(.5)} & {\scriptsize PMA(1)} & {\scriptsize PAR(0)} & {\scriptsize PAR(.2)} &  {\scriptsize PAR(.5)} & {\scriptsize PMA(.5)} & {\scriptsize PMA(1)} \\ \hline
$5$  & .058& .058& .062& .052& .062& .112& .104& .124& .110& .112   \\ \hline
$10$ & .050& .046& .060& .054& .050& .100& .112& .130& .124& .124   \\ \hline
$20$ & .056& .058& .056& .064& .056& .106& .108& .126& .112& .128   \\ \hline
$30$ & .056& .048& .068& .064& .046& .108& .104& .133& .126& .116   \\ \hline
\end{tabular}
\end{center}
\caption{{\it Simulated type I errors   of the test \eqref{boottest} for parallelism  of the mean functions   in the PAR- and PMA-model. The errors in (\ref{eps}) are   $\sqrt{2/3}t_6$-distributed and $\alpha$ represents the nominal level of the test.}}\label{tab:4}
\end{table}

Next, we investigate the finite sample performance of the test  \eqref{boottest} 
under the alternative. To this end, we let 
$$
g_1(u)=u^2+(b-1)u+c_1~\mbox{  and }~~g_i(u)=u^2-u+c_i, ~~i=2,3,\cdots, r, 
$$
where $c_i$, $i=1,2,\cdots, r$ are i.i.d. standard normal random variables, $b\ge 0$ is a constant that controls the magnitude of deviation from the null hypothesis. In particular the choice  $b=0$ corresponds to parallel mean functions and  larger values for $b$ indicate  a more substantial deviation from the null hypothesis. Further observe that there is only one curve that is not parallel to the rest. For the  simulations of the rejection probabilities under the alternative, we investigate a PAR($0.2$) model with $n=200$ observations.  The simulated rejection rates as a function of $b$ are plotted in Figure \ref{Figure1} for dimensions $r=20$ and $30$, respectively.  We observe from Figure \ref{Figure1} that the simulated power of the test 
\eqref{boottest} increases very fast as $b$ increases. In particular, the simulated power of the test becomes reasonably high when $b$ is larger than $0.15$. This observation is consistent with our theoretical finding that the parallelism test is able to detect alternatives converging to the null hypothesis with a rate $\sqrt{\log r/n}$. Additionally, we find that the simulated power for the dimension  $r=20$ is higher than for    dimension $r=30$. This observation is also consistent with the theoretical  $\sqrt{\log r/n}$ rate for the detection of local alternatives.

\begin{figure}[t]
	\centering
	\includegraphics[width=7cm,height=17cm, angle=90]{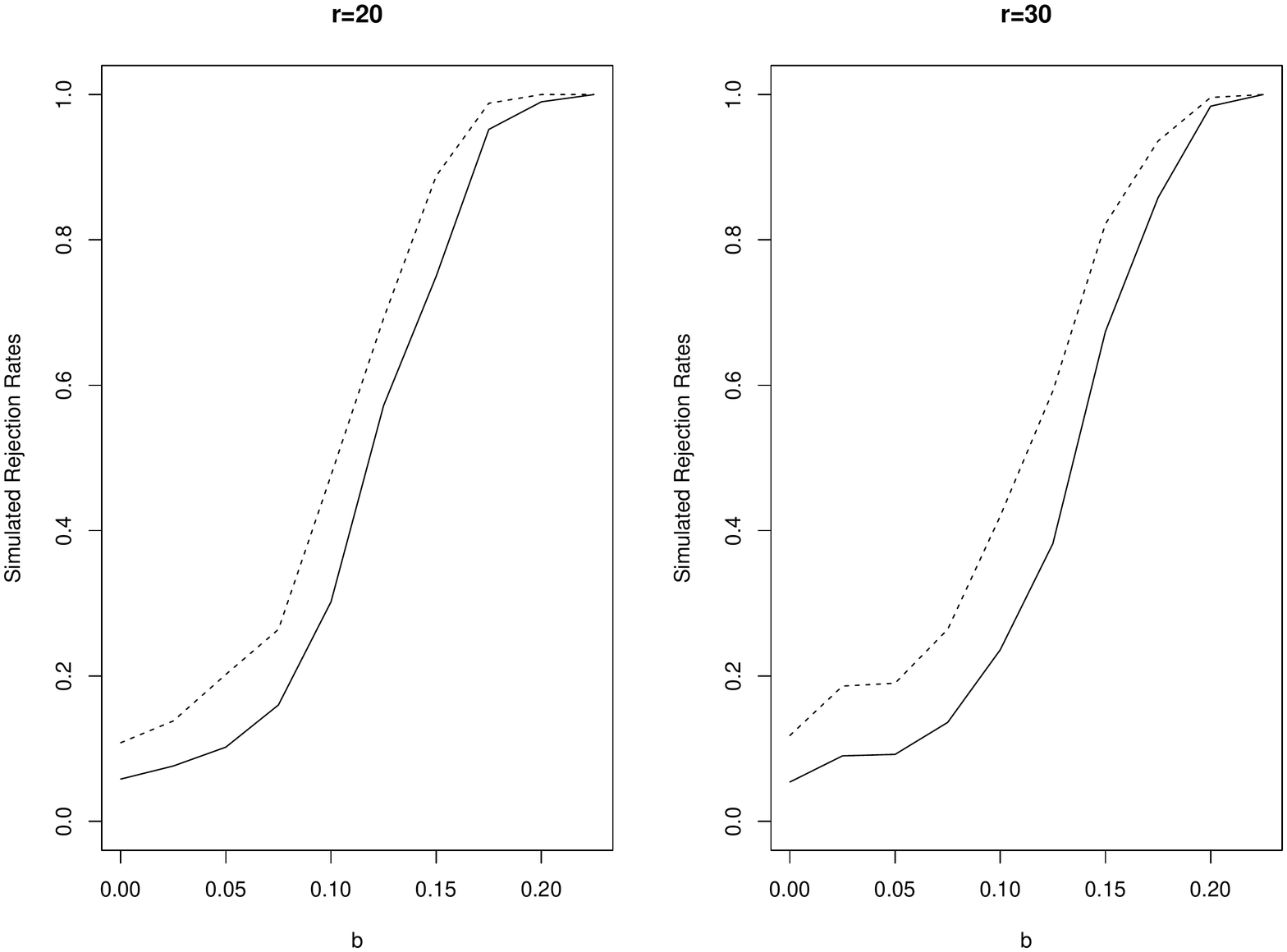}
	\vspace{-1cm}
	\caption{{\it Simulated rejection rates of the  test \eqref{boottest} for the hypothesis of  parallel mean functions  
	with $r=20$ and $30$. Solid and dashed lines represent the rejection rates under nominal levels $0.05$ and $0.1$, respectively.}}
	\label{Figure1}
	\end{figure}

\subsection{Analysis of Canadian temperature data}
\label{sec52}
In this section we analyze daily mean temperature records from 1902 to 2018 in 15 representative Canadian cities as an empirical illustration of our methodology. The cities chosen are Calgary, Charlottetown, Edmonton, Halifax, Hamilton, Kitchener, London, Montreal, Ottawa, Quebec, Saskatoon, Toronto, Vancouver, Victoria, and Windsor. The data are publicly available from the government of Canada website: \url{https://climate.weather.gc.ca/historical_data/search_historic_data_e.html}. The daily temperature records of each year are smoothed using local linear kernel smoothing to represent the overall yearly functional pattern of the temperature records. As a result, we obtain a panel functional time series with temporal dimension $n=117$ and spatial dimension $r=15$. Figure \ref{Figure2} below shows 3D plots of the functional time series of mean temperature in Toronto from two different angles. It can be seen that the yearly mean temperature curves are of similar shape which depicts clear seasonal variation. Additionally, substantial temperature variability can be seen from the curves among different years. \\
\begin{figure}
	\centering
	\includegraphics[width=17.5cm,height=9cm]{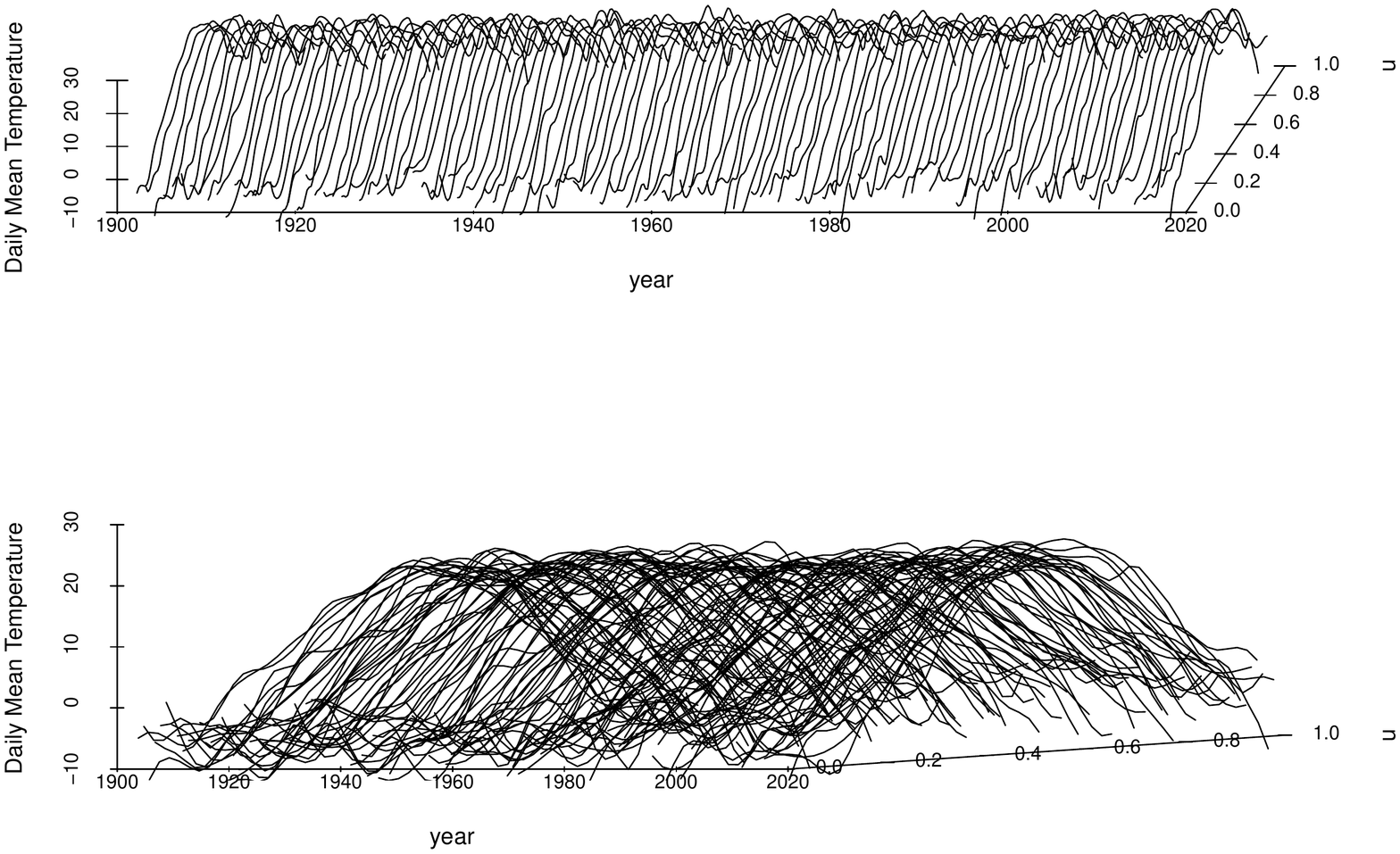}
	\vspace{-1cm}
	\caption{{\it 3D functional time series plots from two different angles for the mean temperature curves of Toronto (1902-2018).}}
	\label{Figure2}
	\end{figure}	
Note that we expect spatial correlation among the panels as some selected cities are close in distance. We are interested in constructing JSCB of the overall yearly pattern of temperature for the 15 cities across Canada. Additionally, we are interested in identifying which cities are similar in yearly temperature patterns in the sense that their temperature curves are parallel to each other.  As the cities chosen in the study are scattered across four climate zones: Atlantic Canada, Great Lakes/St. Lawrence Lowlands, Prairies, and Pacific Coast, we are also interested in checking whether the temperature patterns of the cities can be grouped according to their climate zones.

\begin{figure}
	\centering
	\includegraphics[width=20cm,height=16.5cm,angle=90]{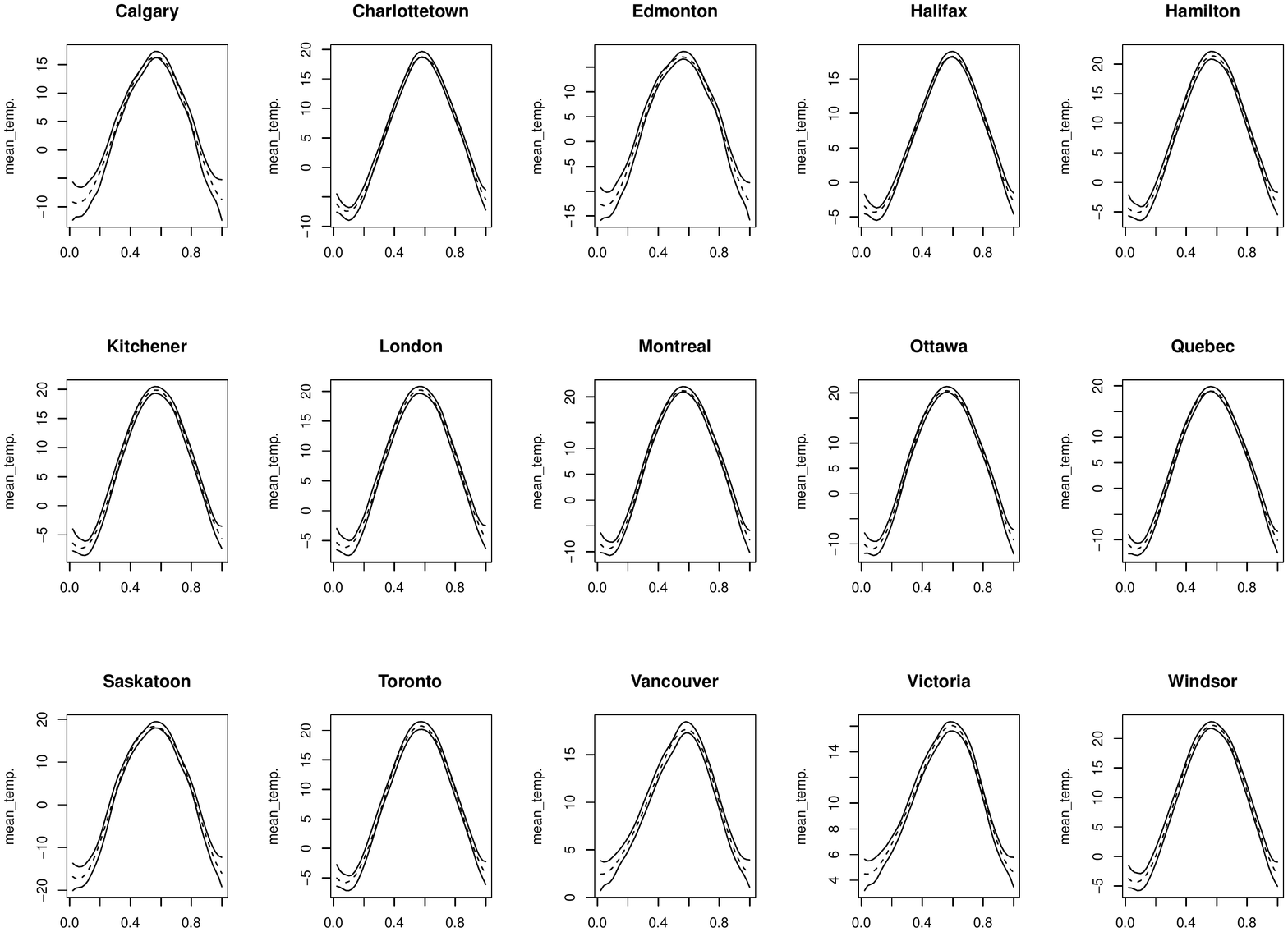}
	\vspace{-1.8cm}
	\caption{{\it 95\% JSCB for the Canadian temperature data. The solid lines represent the upper and lower bounds of the JSCB in (\ref{creg}) with the functions defined in (\ref{hd21}). The dotted lines in the middle represent the fitted harmonic curves $g_i(u)=c_i+\sum_{j=1}^2(a_{i,j}\cos(2\pi j u)+b_{i,j}\sin(2\pi j u))$, $i=1,2,\cdots,15$.}}
	\label{Figure3}
	\end{figure}
	
The $95\%$ JSCB of the yearly temperature pattern of the 15 cities are plotted in
Figure \ref{Figure3}. This JSCB is based on 10,000 bootstrap replications with block size $m=10$ which was selected by the MV method. We can see from the plots that all yearly temperature patterns reflect clear seasonal variation of winter lows and summer highs.  A closer look at the plots shows that Calgary, Edmonton and Saskatoon have significantly wider bands in winter than the other cities. Therefore we conclude that the winter mean temperature in cities of the Prairies zone has more year-to-year variation than the other zones. 

As all temperature curves are similar in shape to the sinusoidal functions, we would like to test the hypothesis 
$$
H_0^k : g_i(u)=c_i+\sum_{j=1}^k(a_{i,j}\cos(2\pi j u)+b_{i,j}\sin(2\pi j u))~, ~i=1,2,\cdots,15.
$$
 Observe that one will encounter multiple testing problems if the cities are tested separately. Here we perform the tests by fitting the harmonic functions to the mean temperature curves and then checking whether all fitted functions can be fully embedded into the 95\% JSCB. When $k=1$, most fitted harmonic curves are not covered by the JSCB. When $k=2$, the fitted harmonic curves are plotted in Figure \ref{Figure3}. It can be seen that all 15 harmonic curves are fully covered by the 95\% JSCB. These results   suggest  that one can use the parametric mean model $g_i(u)=c_i+\sum_{j=1}^2(a_{i,j}\cos(2\pi j u)+b_{i,j}\sin(2\pi j u))$, $i=1,2,\cdots,15$ for the Canadian temperature data.

For the test of parallelism, we follow the implementation steps in Section \ref{sec:tp} with block sizes $m=10$ and $B=10,000$ bootstrap replications. Observe that there are 105 pairs to be compared and that the sample size is $n=117$. Therefore the dimensionality here is almost as large as the sample size. The resulted pairwise test statistics range from 3.0 to 93 and the critical values of the test at 0.1\%, 0.5\%, 1\%, 5\% and 10\% are 7.9, 7.04, 6.61, 5.50 and 5.04, respectively.    
In Table \ref{Figure4} we  display the pairwise 
$p$-values of the parallelism test obtained by the critical values listed above. Note that those pair-wise $p$-values do not suffer from multiple testing problems as \eqref{boottest}  is a joint test of parallelism.
Clearly, the overall test of parallelism is rejected with a very strong evidence. A closer look at Table \ref{Figure4} shows that Hamilton, London, Kitchener, Toronto and Windsor can be grouped together in term of temperature parallelism. And another group of cities with parallel temperature patterns is Quebec, Ottawa and Montreal.  The first group of cities are all in the area of southeast Ontario surrounded by the great lakes. The second group of cities are located near the St. Lawrence River in the St. Lawrence Lowlands. On the other hand, we also observe at some close cities   significantly different temperature patterns. Two such interesting pairs are Vancouver and Victoria and Calgary and Edmonton.

\begin{table}[t]
 \tiny
\begin{center}
\begin{tabular}{|l|r|r|r|r|r|r|r|r|r|r|r|r|r|r|r|}
\hline
& Ham & Lon & Kit & Tor & Win & Que & Ott & Mon & Cal & Edm & Cha & Hal & Vic & Van & Sas \\
\hline
Ham &&&&&&&&&&&&&&& \\
\hline
Lon & $0.2$ &&&&&&&&&&&&&& \\
\hline
Kit & $<0.1$ & $0.3$ &&&&&&&&&&&&& \\
\hline
Tor & $>10$ & $3.6$ & $1.8$ &&&&&&&&&&&& \\
\hline
Win & $>10$ & $>10$ & $>10$ & $>10$ &&&&&&&&&&& \\
\hline
Que & $<0.1$ & $<0.1$ & $<0.1$ & $<0.1$ & $<0.1$ &&&&&&&&&& \\
\hline
Ott & $<0.1$ & $<0.1$ & $<0.1$ & $<0.1$ & $<0.1$ & $5$  &&&&&&&&& \\
\hline
Mon & $<0.1$ & $<0.1$ & $<0.1$ & $<0.1$ & $<0.1$ & $>10$ & $<0.1$  &&&&&&&& \\
\hline
Cal & $0.6$ & $<0.1$ & $0.3$ & $0.1$ & $0.4$ & $<0.1$ & $<0.1$ & $<0.1$ &&&&&&& \\
\hline
Edm & $<0.1$ & $<0.1$ & $<0.1$ & $<0.1$ & $<0.1$ & $<0.1$ & $0.5$ & $<0.1$ & $<0.1$ &&&&&& \\
\hline
Cha & $<0.1$ & $<0.1$ & $<0.1$ & $<0.1$ & $<0.1$ & $<0.1$ & $<0.1$ & $<0.1$ & $<0.1$ & $<0.1$ &&&&& \\
\hline
Hal & $<0.1$ & $<0.1$ & $<0.1$ & $<0.1$ & $<0.1$ & $<0.1$ & $<0.1$ & $<0.1$ & $<0.1$ & $<0.1$ & $<0.1$ &&&& \\
\hline
Vic & $<0.1$ & $<0.1$ & $<0.1$ & $<0.1$ & $<0.1$ & $<0.1$ & $<0.1$ & $<0.1$ & $<0.1$ & $<0.1$ & $<0.1$ & $<0.1$ &&&  \\
\hline
Van & $<0.1$ & $<0.1$ & $<0.1$ & $<0.1$ & $<0.1$ & $<0.1$ & $<0.1$ & $<0.1$ & $<0.1$ & $<0.1$ & $<0.1$ & $<0.1$ & $<0.1$ &&  \\
\hline
Sas & $<0.1$ & $<0.1$ & $<0.1$ & $<0.1$ & $<0.1$ & $<0.1$ & $<0.1$ & $<0.1$ & $<0.1$ & $<0.1$ & $<0.1$ & $<0.1$ & $<0.1$ & $<0.1$ &  \\
\hline
\end{tabular}
\end{center}
\caption{\it Table of pairwise parallelism test $p$-values (in percentage) for the Canadian temperature data.
	\label{Figure4}}
\end{table}

In order to carry out a further investigation, the 95\% JSCB for the mean temperature differences of the two pairs of cities are plotted in Figure \ref{Figure5}.   Vancouver and Victoria are both on the southern Pacific coast of Canada and are within 100 kilometers in distance.  However,  Figure \ref{Figure5} shows that Victoria is 1-2 degrees Celsius warmer than Vancouver in the winter and 1.5 to 2.5 degrees Celsius cooler than Vancouver in the summer. Therefore the two cities are significantly different in yearly temperature pattern. This may be due to the fact that Victoria is located in the Vancouver Island and the island adjusts the temperature significantly. Calgary and Edmonton are major cities of Alberta. Both cities are in the Prairies with Edmonton about 300 kilometers north of Calgary.  Figure \ref{Figure5} shows that Edmonton is significantly colder than Calgary in the winter but significantly warmer in the summer. Interestingly, we observe from Table  \ref{Figure4} that the temperature pattern of Edmonton has some similarities (highest $p$-value around 0.5\%) to that of the St. Lawrence Lowland group while the temperature pattern of Calgary has some similarities (highest $p$-value around 0.5\%) to that of the great lakes group. It is unknown why this is the case and further investigations with geography and climatology knowledge are needed.


\begin{figure}[h]
	\centering
	\includegraphics[width=16cm,height=8cm]{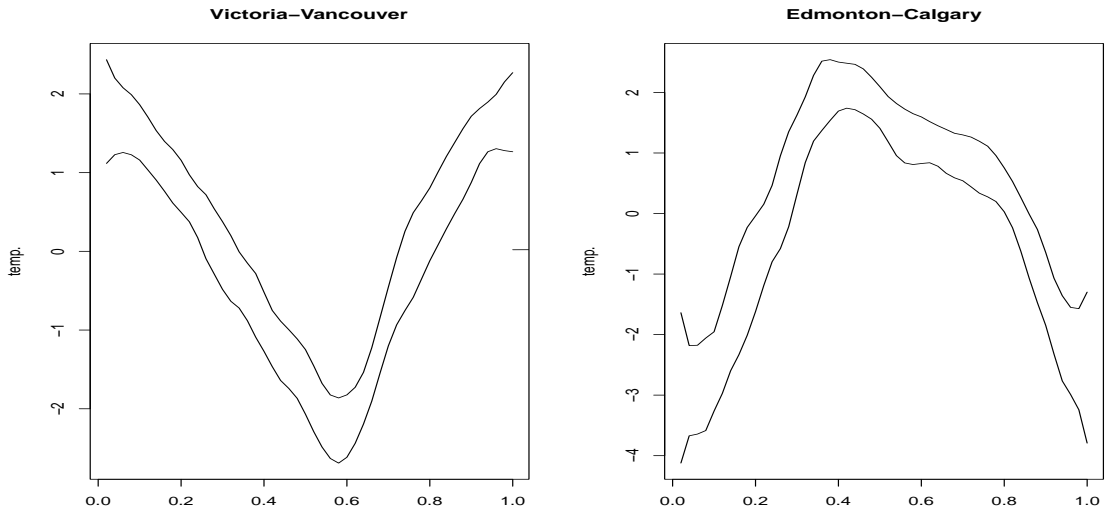}
	\vspace{-1cm}
	\caption{{\it 95\% JSCB for the mean yearly temperature differences between Victoria and Vancouver (left panel) and Edmonton and Calgary  (right panel).
		\label{Figure5}}
	}
	\end{figure}	

\bigskip

{\bf Acknowledgements} This work has been supported in part by the
Collaborative Research Center ``Statistical modeling of nonlinear
dynamic processes'' (SFB 823, Teilprojekt A1,C1) of the German Research Foundation
(DFG) and NSERC of Canada (fund number : 489079). We gratefully acknowledge Professors  Alois Kneip, Dominik Liebl and  Hans Georg M\"uller for some helpful discussions on  related literature.

\newpage
~~\\
\newpage

\centerline{\large\bf Supplemental Material for ``Statistical Inference for High Dimensional}
\centerline{\large\bf  Panel Functional Time Series''}
	\centerline{\sc Zhou Zhou
		and Holger Dette}
		
	\centerline{\sc \small University of Toronto and Ruhr-Universit\"at Bochum  }

\bigskip
\font\n=cmcsc10

\centerline{\today}

\begin{abstract}
This supplemental material contains proofs of the theoretical results of the paper.

\end{abstract}

In this supplemental material, the symbol $C$ denotes a positive finite constant which may vary from place to place. 


\subsection*{Proof of the results in Section \ref{sec2}} 

\begin{proof}[\bf Proof of Proposition \ref{prop:1}]
Using integration by parts, we have  for the Fourier coefficients in (\ref{fcoeff}) and $k\ge 1$,
\begin{eqnarray*}
a_{i,k}
=\frac{2}{(k\pi)^2}\Big(-\int_{0}^1X_i''(u)\cos(k\pi u)\,du+X_i'(1)\cos(k\pi)-X_i'(0)\Big).
\end{eqnarray*}
Note that $\|X_i'(1)\|_q\le\|X_i'(0)\|_q+\int_{0}^1\|X_i''(u)\|_q\,du<\infty$. As a result $\|a_{i,k}\|_q\le C/k^2$.  Hence $\|a^*_{i,k}\|_q\le C/k^2$ as $a_{i,k}$ and $a^*_{i,k}$ are identically distributed. On the other hand,
\begin{eqnarray*}
\|a_{i,k}-a^*_{i,k}\|_q&\le& \int_{0}^1\|H(u,\FF_i)-H(u,\FF^*_i)\|_q|\cos(k\pi u)|\,du\cr
&\le& \sup_{u\in[0,1]}\|H(u,\FF_i)-H(u,\FF^*_i)\|_q=\delta_X(i,q),
\end{eqnarray*}
and the proposition follows.
\end{proof}

\begin{proof}[\bf Proof of Proposition \ref{prop:2}]
Let $A_{n,k}=\sum_{i=1}^na_{i,k}/\sqrt{n}$.
 Observe that $\E (a_{j,k})=0$. For any $j
 \in{\mathbb Z}$ and a random variable $Z$, let 
 $${\cal P}_j(Z)=\E(Z|\FF_j)-\E(Z|\FF_{j-1}).$$ 
 We have
\begin{eqnarray*}
\|A_{n,j}\|_q&\le& C \sum_{l=0}^\infty \|{\cal P}_0(a_{l,j})\|_q\cr 
&\le& C \sum_{l=0}^\infty \|a_{l,j}-a^*_{l,j}\|_q
\le C \Big [\sum_{l=0}^{\lf j^{2/\beta}\rf}j^{-2}+ \sum_{l=\lf j^{2/\beta}\rf+1}^{\infty}l^{-\beta}\Big ]\le Cj^{2(1-\beta)/\beta}
\end{eqnarray*}
by Proposition \ref{prop:1}, Theorem 1 (i), (ii) of \cite{wu2005nonlinear} and Theorem 1(iii) of \cite{wu2007}, where we 
utilized the facts that $q\ge 2$ and $\sum_{l=0}^\infty \|a_{l,j}-a^*_{l,j}\|_q<\infty$. Observing the representation
$$
S_n(u)-S_{n}(k,u)=\sum_{j=k+1}^\infty A_{n,j}\cos(j\pi u)
$$
 we obtain
 $$
\Big  \|\sup_{u\in[0,1]}|S_n(u)-S_{n}(k,u)|  \Big \|_q\le \sum_{j=k+1}^\infty \|A_{n,j}\|_q\le C\sum_{j=k+1}^\infty j^{2(1-\beta)/\beta}\le Ck^{(2-\beta)/\beta},
 $$
and the assertion of the proposition follows.
\end{proof}

\begin{proof}[\bf  Proof of Proposition \ref{prop:3}]
Observe that $S_{n}(k,u)-S_{n}(k,t_{i,n})=\int_{t_{i,n}}^uS_n'(k,s)\,ds$, where $S_n'(k,s)$ is the derivative of $S_n(k,s)$ with respect to $s$. Therefore
$$\sup_{u\in [t_{i,n},t_{i+1,n}]}|S_{n}(k,u)-S_{n}(k,t_{i,n})|\le \int_{t_{i,n}}^{t_{i+1,n}}|S_{n}'(k,s)|\,ds,$$
which implies
\begin{eqnarray}\label{prop3eq1}
\Big \|\sup_{u\in [t_{i,n},t_{i+1,n}]}|S_{n}(k,u)-S_{n}(k,t_{i,n})| \Big \|_q\le \int_{t_{i,n}}^{t_{i+1,n}}\|S_{n}'(k,s)\|_q\,ds.
\end{eqnarray}
Now for any $s\in [0,1]$, $-S_n'(k,s)=\pi \sum_{j=1}^k A_{n,j}j\sin(j\pi s)$. By the same calculations as those in the proof of Proposition \ref{prop:2}  we have
\begin{eqnarray}\label{prop3eq2}
\|S_{n}'(k,s)\|_q\le \sum_{j=1}^k j\|A_{n,j}\|_q\le C\sum_{j=1}^kj \cdot j^{2(1-\beta)/\beta}\le Ck^{2/\beta}.
\end{eqnarray}
Combining (\ref{prop3eq1}) and (\ref{prop3eq2}), we obtain that
$$ \Big \|\sup_{u\in [t_{i,n},t_{i+1,n}]}|S_{n}(k,u)-S_{n}(k,t_{i,n})| \Big \|_q\le  Ck^{2/\beta}/l_n$$
for any $i$.
Hence
$$
\Big \|\max_{0\le i \le l_n-1}\sup_{u\in [t_{i,n},t_{i+1,n}]}|S_{n}(k,u)-S_{n}(k,t_{i,n})| \Big \|_q\le  Ck^{2/\beta}/l^{1-1/q}_n
$$
by an $L^q$ maximum inequality.
\end{proof}

\subsection*{Proof of the results in Section \ref{sec3}} 

Recall the definition (\ref{snvec}) and define
$$
  \bS_{n}(k,u)= \frac {1}{\sqrt{n}} \sum_{i=1}^n\bX_i(k,u) ,
$$
  where
  $$
  \bX_i(k,u)=\sum_{j=1}^k \ba_{i,j}\cos(j\pi u)
  $$
  and $\ba_{i,j}=2\int_{0}^1\bX_i(u)\cos(j\pi u)\,du$ for $j\ge 1$; $\ba_{i,0}=\int_{0}^1\bX_i(u)\,du$. Let
  \begin{equation}\label{snveck}
  \bS^N_{n}(k,u)=\sum_{j=0}^k\bN_{n,j}\cos(j\pi u)
  \end{equation}
  and denote the $j$-th entry of $\bS_{n}(k,u)$ (resp. $\bS^N_{i}(k,u)$) by  $S_{n,j}(k,u)$ (resp. $S^N_{n,j}(k,u)$).

\begin{lemma}\label{lem:2}
Denote the dependence measure of $\{X_j(k,\cdot)\}_{j \in \mathbb{Z}}$ by $\delta_{H,k}(i,q)$. If $\delta_H(i,q)=O(i^{-\beta})$, then
$$\delta_{H,k}(i,q)=O(i^{-\beta/2+\epsilon})\quad\mbox{for any } \epsilon>0.$$
\end{lemma}
\begin{proof}
By definition, we have
$$\delta_{H,k}(i,q)=\Big \|\sum_{j=0}^k(a_{i,j}-a_{i,j}^*)\cos(j\pi u)\Big \|_q,$$
and Proposition \ref{prop:1} implies that $\|a_{i,j}-a_{i,j}^*\|_q \le C\min(j^{-2},i^{-\beta})\le Cj^{-2 (1/2 + \alpha)}i^{-\beta (1/2 - \alpha)}$  for an arbitrarily small $\alpha>0$. We obtain
\begin{eqnarray*}
\delta_{H,k}(i,q)\le \sum_{j=0}^k\|a_{i,j}-a_{i,j}^*\|_q\le \sum_{j=0}^kCj^{-2(1/2 + \alpha)}i^{-\beta (1/2 - \alpha)}\le Ci^{-\beta/2+\alpha\beta}.
\end{eqnarray*}
As $\alpha$ can be made arbitrarily small, the lemma follows.
\end{proof}

\begin{lemma}\label{lem:1}
Under the assumptions of Theorem \ref{thm:1}, we have
\begin{eqnarray*}
\sup_{x\in \R} \Big |\p  \Big [\max_{1\le j\le r}\max_{1\le i\le l_n}|S_{n,j}(k,t_{i,n})|\le x  \Big ]-\p  \Big [\max_{1\le j\le r}\max_{1\le i\le l_n}|S^N_{n,j}(k,t_{i,n})|\le x  \Big  ] \Big |\rightarrow 0.
\end{eqnarray*}
\end{lemma}
\begin{proof}
Let
\begin{equation}\label{h1}
  \tilde \bX_{i,k} =(\bX^\top_{i}(k,t_{1,n}),\cdots,\bX^\top_{i}(k,t_{l_n,n}))^\top
     \end{equation}
     denote the   vector of length $rl_n$ in which the $l$-th block contains the coordinates of the vector ${\bf X}_i(k,t_{l,n})$ and define  $$\tilde{\bS}_{k,n}=\frac {1}{\sqrt{n}}\sum_{j=1}^n\tilde\bX_{i,k} $$ as the corresponding standardized sum.\\
Finally, recall the notation of $\bS^N_n(k,\cdot)$ in (\ref{snveck}),   define $$\tilde{\bS}^N_{k,n}=\big([\bS^N_{n}(k,t_{1,n})]^\top,\cdots,[\bS^N_{n}(k,t_{l_n,n})]^\top\big)^\top$$ and note that $\tilde{\bS}_{k,n}$ and $\tilde{\bS}^N_{k,n}$ share the same covariance structure.  By Proposition \ref{prop:2}, we have
$\|S_{n,j}(u)-S_{n,j}(k,u)\|_2\rightarrow 0$ uniformly in $j$ and $u$ if $k\rightarrow\infty$. Observing assumption (\ref{ass3}) we conclude that, uniformly in $j$ and $u$, $\E(S^2_{n,j}(k,u))\ge \delta/2$ for sufficiently large $n$. Now pick $l_n\asymp n^{\theta_2}$ and $k\asymp n^{\theta_3}$ with positive constants $\theta_2$ and $\theta_3$. By   assumption (\ref{ass4}), we can choose $\theta_2$ and $\theta_3$ such that
\begin{eqnarray*}
\left\{\begin{array}{l}\theta_1+\theta_2<q/2-1,\\ \theta_1/q+[(2-\beta)/\beta]\theta_3<0,\\ \theta_1/q+(1/q-1)\theta_2+(2/\beta)\theta_3<0.\end{array}\right.
\end{eqnarray*}

Let  $\Delta=\sum_{j=-\infty}^\infty\E(\tilde\bX_{i,k}\tilde\bX^\top_{i-j,k})$ denote the long-run-covariance of the time series $\{ \tilde{\bf X}_{i,k}\}_{i \in \mathbb{Z}}$ defined by (\ref{h1}). 
Then as $\theta_1+\theta_2<q/2-1$ and $\max_{1\le k\le r}\delta_{{H_j},k}(i,q)=O(i^{-\beta/2+\epsilon})$ 
for an arbitrarily small positive $\epsilon$ (see Lemma \ref{lem:2}), we can easily check that the assumptions of part (i) of Theorem 3.2 of \cite{zhangwu2017}  hold (define $\alpha$ therein
as  $\beta/2-1-\epsilon$ and note that $\alpha>1/2-1/q$ since $\beta>3$). To provide more details, we denote by $X_{\cdot j}$ the $j$-th component process of $\tilde \bX_{i,k}$, $j=1,2,\cdots, rl_n,$ and   
check the magnitudes of the quantities 
 $\delta_{i,q,j}$ $\Delta_{m,q,j}$, $\|X_{\cdot j}\|_{q,\alpha}$, $\Psi_{q,\alpha}$, $\Upsilon_{q,\alpha}$, $\omega_{i,q}$, $\Omega_{m,q}$, $\||X_{\cdot}|_{\infty}\|_{q,\alpha}$, $\Theta_{q,\alpha}$, $L_1$, $L_2$, $N_1$, $N_2$, $W_1$ and $W_2$ defined in \cite{zhangwu2017}.  By Lemma \ref{lem:2}, $\delta_{i,q,j}=O(i^{-\beta/2+\epsilon})$ for any $j$ which implies  $\Delta_{m,q,j}=O(m^{-\beta/2+1+\epsilon})$.
Elementary but tedious calculations using the assumptions of Theorem \ref{thm:1} yield that 
$$
\|X_{\cdot j}\|_{q,\alpha}=O(1)~,~~\Psi_{q,\alpha}=O(1)~, ~~\Upsilon_{q,\alpha}=O(p^{1/q})
$$
 (note that here $p=O(n^{\theta_1+\theta_2})$), 
 \begin{eqnarray*} 
 \omega_{i,q}&=& O(p^{1/q}i^{-\beta/2+\epsilon})~, ~\Omega_{m,q}=O(p^{1/q}m^{-\beta/2+1+\epsilon})~, ~
  \Theta_{q,\alpha}=O(p^{1/q})~, \\
L_1 &= & o(1)~, ~L_2=O(\log^{2/\alpha} p)~,~   \||X_{\cdot}|_{\infty}\|_{q,\alpha}=O(p^{1/q}) ~,  \\
W_1&=&O(\log^7(pn))~, ~W_2=O(\log^4(pn)),
 \end{eqnarray*}
 and  $N_1\ge c(n/\log p)^{q/2}/p$, $N_2\ge cn/\log^2 p$, where $c$ is a positive constant. Therefore the assumptions of part (i) of Theorem 3.2 of \cite{zhangwu2017}  hold provided that $\theta_1+\theta_2<q/2-1$ and $\epsilon$ is sufficiently small. Thus we obtain
\begin{eqnarray} \label{hd1}
\sup_{x\in \R}\Big |\p[|\tilde{\bS}_{k,n}|_\infty\le x]-\p[|\Delta^{1/2}{\bf G}|_\infty\le x] \Big |\rightarrow 0,
\end{eqnarray}
where $| \cdot |_\infty$ denotes the maximum norm on $\mathbb{R}^{rl_n}$ and ${\bf G}$ is a standard $rl_n$ dimensional Gaussian vector. In fact, Theorem 3.2 in \cite{zhangwu2017}  performs a normalization and it requires that the long-run variance of each $S_{n,j}(k,t_{i,n})$ equals 1, $j=1,2,\cdots, r$, $i=1,2,\cdots, l_n$. But their arguments easily apply to the case of arbitrary marginal long-run variances with a positive lower bound.  Now observe that
\begin{eqnarray*}
\Delta_n:=\E([\tilde{\bS}_{k,n}\tilde{\bS}^\top_{k,n}])=\sum_{j=-n}^n(1-|j|/n)\E(\tilde\bX_{i,k}\tilde\bX^\top_{i-j,k}),
\end{eqnarray*}
and that Lemma 6 of \cite{zhou2014}  and Lemma \ref{lem:2} yield
\begin{eqnarray*}
|\E[X_{i,p,k}X_{i-j,s,k}]|\le Cj^{-\beta/2+\epsilon},\quad 1\le p\le rl_n, 1\le s\le rl_n
\end{eqnarray*}
for any $\epsilon>0$, where $X_{i,p,k}$ is the $p$-th entry of $\tilde\bX_{i,k}$. Hence simple calculations yield for the entries of the matrices $\Delta_n$ and $\Delta$ that
$$\max_{1\le i,j\le lr_n}|\Delta_n(i,j)-\Delta(i,j)|=O(n^{-\beta/2+\epsilon+1}).$$  By the Gaussian comparison inequality in Lemma 3.1 of \cite{chechekato2013}, we have
\begin{eqnarray*}
\sup_{x\in \R} \Big |\p[|\tilde{\bS}^N_{k,n}|_\infty\le x]-\p[|\Delta_n^{1/2}{\bf G}|_\infty\le x] \Big |\le C(n^{(-\beta/2+\epsilon+1)/3}[\log (rl_n / n^{-\beta/2+\epsilon+1} ]^{2/3}\rightarrow 0,
\end{eqnarray*}
which implies (observing (\ref{hd1}))
$$\sup_{x\in \R}|\p[|\tilde{\bS}_{k,n}|_\infty\le x]-\p[\tilde{\bS}^N_{k,n}|_\infty\le x]|\rightarrow 0.$$
And this is exactly the claim of the lemma.
\end{proof}

\begin{proof}[\bf Proof of Theorem \ref{thm:1}]
 Proposition \ref{prop:2} and an application of an $L^q$ maximal inequality yield
\begin{eqnarray*}
\Big \|\max_{1\le j\le r}\sup_{0\le u\le 1}|S_{n,j}(u)-S_{n,j}(k,u)| \Big \|_q\le Cr^{1/q}k^{(2-\beta)/\beta}.
\end{eqnarray*}
Since $\theta_1/q+[(2-\beta)/\beta]\theta_3<0$, the right hand side of the above inequality converges to $0$ faster than $n^{-\alpha_1}$ for some positive $\alpha_1$. By the triangular and Markov inequalities, for $\delta_{1,n}:=\sqrt{r^{1/q}k^{(2-\beta)/\beta}}$, we have
\begin{eqnarray*}
\p\Big(\max_{1\le j\le r}\sup_{0\le u\le 1}|S_{n,j}(u)|\le x \Big)&\le& \p \Big (\max_{1\le j\le r}\sup_{0\le u\le 1}|S_{n,j}(k,u)|\le x+\delta_{1,n}\Big)\cr
&+&\p\Big(\max_{1\le j\le r}\sup_{0\le u\le 1}|S_{n,j}(u)-S_{n,j}(k,u)|\ge \delta_{1,n}\Big)\cr
&\le& \p\Big(\max_{1\le j\le r}\sup_{0\le u\le 1}|S_{n,j}(k,u)|\le x+\delta_{1,n}\Big)+\epsilon_{1,n},
\end{eqnarray*}
where $\epsilon_{1,n}= Crk^{q(2-\beta)/\beta}/\delta^q_{1,n}\rightarrow 0$ with some polynomial rate. Similarly, since $\theta_1/q+(1/q-1)\theta_2+(2/\beta)\theta_3<0$, it follows that $r^{1/q}k^{2/\beta}/l^{1-1/q}_n$ converges to $0$ faster than $n^{-\alpha_2}$ for some positive $\alpha_2$. By Proposition \ref{prop:3}, we have, for $\delta_{2,n}:=\sqrt{r^{1/q}k^{2/\beta}/l^{1-1/q}_n}$,
\begin{eqnarray*}
\p\Big(\max_{1\le j\le r}\sup_{0\le u\le 1}|S_{n,j}(k,u)|\le x+\delta_{1,n}\Big)&\le&\p\Big(\max_{1\le j\le r}\sup_{1\le i\le l_n}|S_{n,j}(k,t_{i,n})|\le x+\delta_{1,n}+\delta_{2,n}\Big)+\epsilon_{2,n},
\end{eqnarray*}
where $\epsilon_{2,n}=C\{\delta_{2,n}/[r^{1/q}k^{2/\beta}/l^{1-1/q}_n]\}^{-q}\rightarrow 0$ with some polynomial rate. Now by Lemma \ref{lem:1}, we obtain
\begin{eqnarray*}
\p\Big (\max_{1\le j\le r}\sup_{1\le i\le l_n}|S_{n,j}(k,t_{i,n})|\le x+\delta_{1,n}+\delta_{2,n}\Big )-\p\Big(\max_{1\le j\le r}\sup_{1\le i\le l_n}|S^N_{n,j}(k,t_{i,n})|\le x+\delta_{1,n}+\delta_{2,n})\rightarrow 0
\end{eqnarray*}
uniformly in $x$. Denote the latter uniform convergence rate by $\epsilon_{3,n}$. Furthermore, by Nazarov's anti-concentration inequality (\cite{nazarov2003}),
\begin{eqnarray*}
\p\Big(\max_{1\le j\le r}\sup_{1\le i\le l_n}|S^N_{n,j}(k,t_{i,n})|\le x+\delta_{1,n}+\delta_{2,n}\Big)&-&\p\Big(\max_{1\le j\le r}\sup_{1\le i\le l_n}|S^N_{n,j}(k,t_{1,n})|\le x\Big)\cr
&\le& C(\delta_{1,n}+\delta_{2,n})\sqrt{\log(rl_n)}.
\end{eqnarray*}
Following the same arguments as above, we have
\begin{eqnarray*}
\p\Big(\max_{1\le j\le r}\sup_{1\le i\le l_n}|S^N_{n,j}(k,t_{i,n})|\le x\Big)&-& \p\Big(\max_{1\le j\le r}\sup_{0\le u\le 1}|S^N_{n,j}(u)|\le x\Big)\cr
&\le& \epsilon_{1,n}+\epsilon_{2,n}+ C(\delta_{1,n}+\delta_{2,n})\sqrt{\log(rl_n)}.
\end{eqnarray*}
Therefore we conclude that
\begin{eqnarray*}
\p\Big[\max_{1\le j\le r}\sup_{0\le u\le 1}|S_{n,j}(u)|\le x\Big]&-&\p\Big[\max_{1\le j\le r}\sup_{0\le u\le 1}|S^N_{n,j}(u)|\le x\Big]\cr
&\le& C[ \epsilon_{1,n}+\epsilon_{2,n}+ \epsilon_{3,n}+(\delta_{1,n}+\delta_{2,n})\sqrt{\log(rl_n)}]\rightarrow 0
\end{eqnarray*}
uniformly in $x$. Similarly, we also obtain the lower bound
\begin{eqnarray*}
\p\Big[\max_{1\le j\le r}\sup_{0\le u\le 1}|S_{n,j}(u)|\le x\Big]&-&\p\Big[\max_{1\le j\le r}\sup_{0\le u\le 1}|S^N_{n,j}(u)|\le x\Big]\cr
&\ge& -C[ \epsilon_{1,n}+\epsilon_{2,n}+\epsilon_{3,n}+ (\delta_{1,n}+\delta_{2,n})\sqrt{\log(rl_n)}]\rightarrow 0
\end{eqnarray*}
uniformly in $x$. Hence we conclude that Theorem \ref{thm:1} holds.
\end{proof}

Recall the definition of the vector   $\bT_{i,m}(u)$ in (\ref{Tu}) and
denote its $j$-th entry  by $T_{i,j,m}(u)$. Similarly define $\Phi_{m,j}(u)$ as the $j$-th component of the vector $\mathbf{\Phi}_m$ in (\ref{phim}).
\begin{lemma}\label{lem:3}
Under the assumptions of Theorem \ref{thm:2}, there exists a sequence of events $E_{1,n}$ such that $\p(E_{1,n})\rightarrow 1$. On $E_{1,n}$, we have
\begin{eqnarray*}
\sup_{x\in\R}\Big|\p\Big(\max_{1\le j\le r}\sup_{1\le i\le l_n}|\Phi_{m,j}(t_{i,n})|\le x\Big|\{\bX_i(u)\}_{i=1}^n\Big)-\p\Big(\max_{1\le j\le r}\sup_{1\le i\le l_n}|S^N_{n,j}(t_{i,n})|\le x\Big)\Big|\rightarrow 0.
\end{eqnarray*}
\end{lemma}

\begin{proof}
By Theorem 4 of \cite{zhou2013}, we have, for each $t_{l,n},$ $1\le l\le l_n$ and $j$, $1\le j\le r$,
\begin{eqnarray*}
\Big \|\sum_{i=1}^{n-m}\frac{m}{n-m}\big[T_{i,j,m}(t_{l,n})-S_{n,j}(t_{l,n})/\sqrt{n}\big]^2-\tilde{\Sigma}_{(l-1)r+j,(l-1)r+j}\Big \|_{q'}\le C(\sqrt{m/n}+1/m),
\end{eqnarray*}
where $q'=q/2$, $\tilde{\Sigma}_{i,j}$ denotes the $(i,j)$-th entry of the matrix $\tilde{\Sigma}$
$$
\tilde{\Sigma}=\sum_{j=-\infty}^\infty \E[\bar\bX_i-\E \bar\bX_i][\bar\bX^\top_{i-j}-\E \bar\bX^\top_{i-j}],$$
 and $\bar\bX_i=(\bX^\top_i(t_{1,n}),\cdots,\bX^\top_i(t_{l_n,n}) )^\top$.

Let $\tilde{\Sigma}_n=\E (\bS^N_n[\bS^N_n]^\top)$, where $\bS^N_n=([\bS^N_n(t_{1,n})]^\top,\cdots, [\bS^N_n(t_{l_n,n})]^\top)^\top$. Observe that for each fixed $l$ and $j$, the $k$-th dependence measure of the sequence $\{X_{i,j}(t_{i,n})\}_{i=1}^n$ decays at the rate $O(k^{-\beta})$.
Hence by similar arguments as those given in the proof of Lemma \ref{lem:1}, we obtain
\begin{eqnarray*}
\max_{1\le i,j\le rl_n}|\tilde{\Sigma}_{n,i,j}-\tilde{\Sigma}_{i,j}|\le Cn^{-\beta+1},
\end{eqnarray*}
and we conclude that
\begin{eqnarray}\label{eq:1}
\Big \|\sum_{i=1}^{n-m}\frac{m}{n-m}[T_{i,j,m}(t_{l,n})-S_{n,j}(t_{l,n})/\sqrt{n}]^2-\tilde{\Sigma}_{n,(l-1)r+j,(l-1)r+j}\Big \|_{q'}\le C(\sqrt{m/n}+1/m).
\end{eqnarray}
From these considerations   and by arguments similar to those in  the proof of Theorem 4 of \cite{zhou2013}, we also obtain that
\begin{eqnarray*}
\Big \|\sum_{i=1}^{n-m}\frac{m}{n-m}[T_{i,j,m}(t_{l,n})-\frac{S_{n,j}(t_{l,n})}{\sqrt{n}}][T_{i,h,m}(t_{w,n})-\frac{S_{n,h}(t_{w,n})}{\sqrt{n}}]&-&\tilde{\Sigma}_{n,(l-1)r+j,(w-1)r+h}\Big \|_{q'}\cr&\le& C(\sqrt{m/n}+1/m),
\end{eqnarray*}
for any $1\le j,h\le r$ and $1\le l,w\le l_n$. Therefore,   an $L^{q'}$ maximum inequality yields
\begin{eqnarray*}
\Big\|\max_{1\le a,b\le rl_n}|[\mbox{Cov}(\bar{\bPhi}_{m}|\{\bX_i(u)\}_{i=1}^n)]_{a,b}-\tilde{\Sigma}_{n,a,b}|\Big\|_{q'}\le [rl_n]^{2/q}(\sqrt{m/n}+1/m):=\gamma_{1,n},
\end{eqnarray*}
where
 $\bar{\bPhi}_{m}=(\bPhi^\top_{m}(t_{1,n}),\cdots, \bPhi^\top_{m}(t_{l_n,n}))^\top$. By the assumption $\theta_1<\phi'q^2/[2(q+1)]$, we can choose $l_n=n^{\theta_2}$ such that
\begin{eqnarray*}
\left\{\begin{array}{l}\theta_1/q-\theta_2<0\\ 2(\theta_1+\theta_2)/q<\phi'\end{array}\right.
\end{eqnarray*}
Since $2(\theta_1+\theta_2)/q<\phi'$, we conclude that $\gamma_{1,n}$ converges to zero at a polynomial rate. Hence the event 
$$E_{1,n}:=\max_{1\le a,b\le rl_n}\big|[\mbox{Cov}(\bar{\bPhi}_{m}|\{\bX_i(u) \}_{i=1}^n)]_{a,b}-\tilde{\Sigma}_{n,a,b}\big|\le \sqrt{\gamma_{1,n}}$$
has probability at least $1-\gamma_{1,n}^{q/2}$. On event $E_{1,n}$, we have by Lemma 3.1 of \cite{chechekato2013}  that
\begin{eqnarray*}
\sup_{x\in\R}\Big|\p\Big(\max_{1\le j\le r}\sup_{1\le i\le l_n}|\Phi_{m,j}(t_{i,n})|\le x\Big|\{\bX_i(u)\}_{i=1}^n\Big)&-& \p\Big(\max_{1\le j\le r}\sup_{1\le i\le l_n}|S^N_{n,j}(t_{i,n})|\le x\Big)\Big|\cr
&\le& C\gamma_{1,n}^{1/6}\log (rl_n\gamma^{-1/2}_{1,n})^{2/3}\rightarrow 0,
\end{eqnarray*}
which completes the proof of Lemma \ref{lem:3}.
\end{proof}

\begin{lemma}\label{lem:5}
Assume $m/n\rightarrow 0$. Then under the assumptions of Theorem \ref{thm:2} , we have
\begin{eqnarray*}
\Big \|\sup_u\sum_{i=1}^{n-m}\frac{m}{n-m}[T'_{i,j,m}(u)-S'_{n,j}(u)/\sqrt{n}]^2 \Big \|_{q/2}\le C.
\end{eqnarray*}
\end{lemma}
\begin{proof}
Without loss of generality, we assume $\E(X'_{i,j}(u))=0$. First we write $X'_{i,j}(u):=\sum_{j=0}^\infty b_{i,j,k}\cos(k\pi u)$. Then, similar to the proof of Proposition \ref{prop:1}   utilizing $\|X_{1,j}''(u)-X''_{1,j}(v)\|_q\le C|u-v|$, it follows
\begin{eqnarray*}
\|b_{i,j,k}-b^*_{i,j,k}\|_q\le C\min(1/k^2,i^{-\beta}).
\end{eqnarray*}
Let $B_{j,k}=\sum_{i=1}^nb_{i,j,k}/\sqrt{n}$. Then by the proof of Proposition \ref{prop:2}, we have $\|B_{j,k}\|_q\le Ck^{2(1-\beta)/\beta}$ which implies
\begin{eqnarray*}
\|\sup_{0\le u\le 1}|S'_{n,j}(u)|^2\|_{q/2}\le \big(\sum_{j=0}^\infty \|B_{j,k}\|_q\big)^2\le C.
\end{eqnarray*}
Similarly, we have for any $i$,
\begin{eqnarray*}
m\|\sup_{0\le u\le 1}|T'_{i,j,m}(u)|^2\|_{q/2}\le C.
\end{eqnarray*}
As a result the lemma follows.
\end{proof}


\begin{lemma}\label{lem:4}
Assume $m/n\rightarrow 0$. Then under the assumptions of Theorem \ref{thm:2} and the choice of $\theta_2$ in Lemma \ref{lem:3}, there exists a sequence of events $E_{2,n}$ such that $\p(E_{2,n})\rightarrow 1$, and   on $E_{2,n}$  we have the conditional Orlicz  norm
$$\Big\|\max_j\max_{1\le i\le l_n-1}\sup_{u\in[t_{i,n},t_{i+1,n}]}|\Phi_{m,j}(u)-\Phi_{m,j}(t_{i,n})|\Big\|_{\Psi_2}\rightarrow 0.$$
\end{lemma}
\begin{proof}
Note that, for any $u\in[t_{i,n},t_{i+1,n}]$
\begin{eqnarray}\label{eq:4-1}
|\Phi_{m,j}(u)-\Phi_{m,j}(t_{i,n})|=|\int_{t_{i,n}}^u\Phi'_{m,j}(s)\,ds|\le \int_{t_{i,n}}^{t_{i+1,n}}|\Phi'_{m,j}(s)|\,ds.
\end{eqnarray}
For each $j$ and $s$, we have for the conditional variance
\begin{eqnarray*}
\mbox{Var}(\Phi'_{m,j}(s)|\{\bX_h(u)\}_{h=1}^n)=\sum_{i=1}^{n-m}\frac{m}{n-m}[T'_{i,j,m}(s)-S'_{n,j}(s)/\sqrt{n}]^2,
\end{eqnarray*}
and Lemma \ref{lem:5} implies (using a maximal inequality) that
$$\Big \|\max_j\sup_u\sum_{i=1}^{n-m}\frac{m}{n-m}[T'_{i,j,m}(u)-S'_{n,j}(u)/\sqrt{n}]^2 \Big \|_{q/2}\le Cr^{2/q}.$$
Let $\gamma_{3,n}=r^{2/q+\eta}$ for an arbitrarily small positive constant $\eta$, and define the event
\begin{align*} 
E_{2,n}=\max_j\sup_u\sum_{i=1}^{n-m}\frac{m}{n-m}[T'_{i,j,m}(u)-S'_{n,j}(u)/\sqrt{n}]^2\le \gamma_{3,n}.
\end{align*}
By Markov's inequality $\p(E_{2,n})\ge 1-r^{-\eta q/2}$. On the event $E_{2,n}$, we have that, conditional on $\{\bX_h(u)\}_{h=1}^n$, 
$\mbox{Std}(\Phi'_{m,j}(s))\le \sqrt{\gamma_{3,n}}$ uniformly  in $j$ and $s$, where $\mbox{Std}$ stands for ``standard deviation''. Note that the Orcliz norm is proportional to the standard deviation for centered Gaussian random variables. Therefore, on the event $E_{2,n}$ and conditional on $\{\bX_h(u)\}_{h=1}^n$, we have
$\|\Phi'_{m,j}(s)\|_{\Psi_2}\le \sqrt{\gamma_{3,n}}$ uniformly in $j$ and $s$.
Hence it follows by a maximal inequality that, on the event $E_{2,n}$ and conditional on $\{\bX_h(u)\}_{h=1}^n$, the Orlicz  norm
\begin{eqnarray*}
&&\Big \|\max_j\max_{1\le i\le l_n-1}\sup_{u\in[t_{i,n},t_{i+1,n}]}|\Phi_{m,j}(u)-\Phi_{m,j}(t_{i,n})| \Big \|_{\Psi_2}\cr
&\le& \sqrt{\log (rl_n)}\max_{i,j}\Big \|\sup_{u\in[t_{i,n},t_{i+1,n}]}|\Phi_{m,j}(u)-\Phi_{m,j}(t_{i,n})| \Big \|_{\Psi_2}\cr
&\le& \sqrt{\log (rl_n)}\max_j\sup_s \|\Phi'_{m,j}(s)\|_{\Psi_2}/l_n\le \sqrt{\gamma_{3,n}}\sqrt{\log (rl_n)}/l_n.
\end{eqnarray*}
The right-hand side converges to zero polynomially fast for sufficiently small $\eta$ since $\theta_1/q-\theta_2<0$.

\end{proof}

\begin{proof}[\bf Proof of Theorem \ref{thm:2}]
Note that, for each fixed $j$,
\begin{eqnarray*}
\Big\|\sup_{0\le s\le 1}|[S^N_{n,j}]'(k,s)|\Big\|_{\Psi_2}\le \sum_{l=0}^k l\|N_{n,l,j} \|_{\Psi_2}\le  C\sum_{l=0}^k l\|N_{n,l,j} \|_2\le C\sum_{l=0}^k l\times l^{2(1-\beta)/\beta}\le Ck^{2/\beta},
\end{eqnarray*}
where we have used the fact that the Orcliz norm is proportional to the standard deviation for centered Gaussian random variables. Hence we obtain that
\begin{eqnarray}\label{eq:thm2-1}
\Big\|\max_j\sup_{0\le s\le 1}|[S^N_{n,j}]'(k,s)|\Big\|_{\Psi_2}\le  Ck^{2/\beta}\sqrt{\log r}.
\end{eqnarray}
On the other hand, we have
\begin{eqnarray*}
\Big\|\sup_u|S^N_{n,j}(u)-S^N_{n,j}(k,u)|\Big\|_{\Psi_2}\le C\sum_{l=k+1}^\infty\|N_{n,l,j}\|_{\Psi_2}\le  Ck^{2/\beta-1},
\end{eqnarray*}
which implies
\begin{eqnarray}\label{eq:thm2-2}
\Big\|\max_j\sup_u|S^N_{n,j}(u)-S^N_{n,j}(k,u)|\Big\|_{\Psi_2}\le  Ck^{2/\beta-1}\sqrt{\log r}.
\end{eqnarray}
Combining (\ref{eq:thm2-1}) and (\ref{eq:thm2-2}) and by the triangular inequality, we obtain that
\begin{eqnarray}\label{eq:thm2-3}
\Big\|\max_j\max_{1\le i\le l_n}\sup_{t_i\le u\le t_{i+1}}|S^N_{n,j}(u)-S^N_{n,j}(t_{i,n})|\Big\|_{\Psi_2}\le  C[k^{2/\beta}/l_n+k^{2/\beta-1}]\sqrt{\log r}.
\end{eqnarray}
Choose $k$ diverging to infinity with a slowly enough polynomial rate  such that the right hand side of (\ref{eq:thm2-3}) converges to 0 with a polynomial rate $n^{-\eta_1}$ for some  $\eta_1>0$. We have that, for any $x\in{\mathbb R}$,
\begin{eqnarray*}
\p\Big[\max_{1\le j\le r}\sup_{0\le u\le 1}|S^N_{n,j}(u)|\le x\Big]&-&\p\Big[\max_{1\le j\le r}\max_{1\le i\le l_n}|S^N_{n,j}(t_{i,n})|\le x-\tau_{1,n}\Big]\cr
&\ge& -\p[\max_j\max_{1\le i\le l_n}\sup_{t_{i,n}\le u\le t_{i+1,n}}|S^N_{n,j}(u)-S^N_{n,j}(t_{i,n})|\ge \tau_{1,n}]\cr
&\ge& -n^{-C_0}
\end{eqnarray*}
for some constant $C_0>0$, where $\tau_{1,n}=n^{-\eta_1}\sqrt{\log n}$ and we have  used Markov's inequality in the second inequality above.
Similarly, by Lemma \ref{lem:4}, we have that, on the event $E_{2,n}$,
\begin{eqnarray*}
\p\Big[\max_{1\le j\le r}\sup_{0\le u\le 1}|\Phi_{m,j}(u)|&\le& x ~\Big|\{\bX_i\}_{i=1}^n\Big]-\p\Big[\max_{1\le j\le r}\max_{1\le i\le l_n}|\Phi_{m,j}(t_{i,n})|\le x+\tau_{2,n} ~\Big|\{\bX_i\}_{i=1}^n\Big]\cr
&\le& \p[\max_j\max_{1\le i\le l_n}\sup_{t_{i,n}\le u\le t_{i+1,n}}|\Phi_{m,j}(u)-\Phi_{m,j}(t_{i,n})|\ge \tau_{2,n}|\{\bX_i\}_{i=1}^n]\cr
&\le& n^{-C_1}
\end{eqnarray*}
for some constant $C_1>0$, where $\tau_{2,n}=n^{-\eta_2}\sqrt{\log n}$ for some constant $\eta_2>0$.
Therefore, by Lemma \ref{lem:3} and on the event $E_{1,n}\cap E_{2,n}$, we have 
\begin{eqnarray*}
P_{n}&:=&\p\Big[\max_{1\le j\le r}\sup_{0\le u\le 1}|\Phi_{m,j}(u)|\le x ~\Big|\{\bX_i\}_{i=1}^n\Big]-\p\Big[\max_{1\le j\le r}\sup_{0\le u\le 1}|S^N_{n,j}(u)|\le x\Big]\cr
&\le &\p\Big[\max_{1\le j\le r}\max_{1\le i\le l_n}|\Phi_{m,j}(t_{i,n})|\le x+\tau_{2,n} ~\Big|\{\bX_i\}_{i=1}^n\Big]-\p\Big[\max_{1\le j\le r}\max_{1\le i\le l_n}|S^N_{n,j}(t_{i,n})|\le x-\tau_{1,n}\Big]+O(n^{-\tilde C}),\cr
&\le & \eta_{3,n}+\p\Big[x-\tau_{1,n}<\max_{1\le j\le r}\max_{1\le i\le l_n}|S^N_{n,j}(t_{i,n})|\le x+\tau_{2,n}\Big]+O(n^{-\tilde C}),
\end{eqnarray*}
where $\tilde C=\min(C_0, C_1)$ and $\eta_{3,n}$ is the approximation error in Lemma \ref{lem:3}. By Nazarov's Inequality (\cite{nazarov2003}),
\begin{eqnarray*}
\p\Big[x-\tau_{1,n}<\max_{1\le j\le r}\max_{1\le i\le l_n}|S^N_{n,j}(t_{i,n})|\le x+\tau_{2,n}\Big]\le C(\tau_{2,n}+\tau_{1,n})\sqrt{\log(rl_n)}
\end{eqnarray*}
which converges to 0 at a polynomial rate. Hence, it follows
\begin{eqnarray*}
P_{n}&\le &\eta_{3,n}+O((\tau_{2,n}+\tau_{1,n})\sqrt{\log(rl_n)})+O(n^{-\tilde C})
\end{eqnarray*}
uniformly with respect to $x\in{\mathbb R}$. Note that the right hand side of the above inequality converges to 0 at a polynomial rate. By a similar argument, we obtain on the event $E_{1,n}\cap E_{2,n}$ that
\begin{eqnarray*}
P_{n}& \ge & -\big[\eta_{3,n}+O((\tau_{2,n}+\tau_{1,n})\sqrt{\log(rl_n)})+O(n^{-\tilde C})\big]
\end{eqnarray*}
uniformly with respect to  $x\in{\mathbb R}$, which implies 

\begin{eqnarray*}
\sup_{x\in \R}\Big|\p\Big[\max_{1\le j\le r}\sup_{0\le u\le 1}|\Phi_{m,j}(u)|\le x ~\Big|\{\bX_i\}_{i=1}^n\Big]-\p\Big[\max_{1\le j\le r}\sup_{0\le u\le 1}|S^N_{n,j}(u)|\le x\Big]\Big|\rightarrow 0.
\end{eqnarray*}
\end{proof}

\subsection*{Proofs of the results in Section \ref{sec4}}

\begin{proof}[\bf Proof of Proposition \ref{prop:scb}.] Observe that $\check X_{i,j}(u)$ is a centered process. Hence this proposition follows from Theorem \ref{thm:1} if we can show that $v_j(u)$ is twice continuously differentiable on $[0,1]$ for all $j$. Observe that $g_j(u)=\E(X_{i,j}(u))$. By the assumption that $X_{i,j}(u)$ is twice continuously differentiable,   (\ref{ass0})  and a version of the Dominated Convergence Theorem (c.f. Theorem 1.6.8 of \cite{durrett2010}), we have that $g_j(u)$ is twice continuously differentiable. As a result $X_{i,j}(u)-g_j(u)$ is  twice continuously differentiable a.s..  Since $v^2_j(u)=\E[X_{i,j}(u)-g_j(u)]^2$,  (\ref{ass0})  and the above mentioned version of the Dominated Convergence Theorem imply that  $v_j(u)$ is twice continuously differentiable for all $j$.
\end{proof}

\begin{proof}[\bf Proof of Lemma \ref{lem:hatv}.]
Recall the notation (\ref{h7}) and observe that
\begin{equation}\label{h8}
 \hat{v}^2_j(u)-v^2_j(u)=\frac{1}{n}\sum_{i=1}^n[\check{X}^2_{i,j}(u)-\E\check{X}^2_{i,j}(u)]-\frac{1}{n}\check{S}^2_{n,j}(u).
\end{equation}
Let $Z_{i,j}(u)=\check{X}^2_{i,j}(u)-\E\check{X}^2_{i,j}(u)$. Then $\{Z_{i,j}(u)\}_{i=1}^n$ is a centered functional time series with dependence measures
\begin{eqnarray*}
\delta_{Z_{i,j}}(k,q/2)&=&\sup_{u\in[0,1]}\|Z_{k,j}(u)-Z^*_{k,j}(u)\|_{q/2}\le\sup_{u\in[0,1]}\|\check{X}_{k,j}(u)+\check{X}^*_{k,j}(u)\|_{q}\|X_{k,j}(u)-X^*_{k,j}(u)\|_{q}\cr
&\le& C\delta_{G_j}(k,q)\le Ck^{-\beta}.
\end{eqnarray*}
Therefore, by the arguments in the proof of Lemma \ref{lem:5}, we obtain that
\begin{eqnarray*}
\Big \|\sup_{u\in[0,1]}\frac{1}{n}\Big |\sum_{i=1}^n[\check{X}^2_{i,j}(u)-\E\check{X}^2_{i,j}(u)]\Big | \Big \|_{q/2}\le C/\sqrt{n}
\end{eqnarray*}
and an $L^{q/2}$ maximal inequality yields
\begin{eqnarray*}
\Big \|\max_{1\le j\le r}\sup_{u\in[0,1]}\frac{1}{n}\Big |\sum_{i=1}^n[\check{X}^2_{i,j}(u)-\E\check{X}^2_{i,j}(u)]\Big | \Big \|_{q/2}\le Cr^{2/q}/\sqrt{n}.
\end{eqnarray*}
Similarly, we also obtain that
\begin{eqnarray*}
\Big \|\max_{1\le j\le r}\sup_{u\in[0,1]}\frac{1}{n}\check{S}^2_{n,j}(u)\Big \|_{q/2}\le  Cr^{2/q}/n.
\end{eqnarray*}
The Lemma follows by these   two inequalities and (\ref{h8}).
\end{proof}

\begin{proof}[\bf  Proof of Proposition \ref{prop:tp}]
Recalling the Fourier expansion $\bX_{i}(u)=\sum_{k=0}^\infty \ba_{i,k}\cos(k\pi u)$. It follows that
 $W_{i,j}(u)$ admits  the following cosine representation
\begin{eqnarray}
W_{i,j}(u)=\sum_{k=1}^\infty a_{i,j,k}\cos(k\pi u),
\end{eqnarray}
where $a_{i,j,k}$ is the $j$-th component of $\ba_{i,k}$, $j=1,2,\cdots, r$.
As a result, $W_{i,j,k}(u)$ can be represented in the form
\begin{eqnarray}
W_{i,j,k}(u)=\sum_{c=1}^\infty f_{i,j,k,c}\cos(c\pi u),
\end{eqnarray}
where $ f_{i,j,k,c}=a_{i,j,c}-a_{i,k,c}$. From the proof of Proposition \ref{prop:1}, it follows
\begin{eqnarray}
\max_{j,k}\|f_{i,j,k,c}-f^*_{i,j,k,c}\|_q\le C\min(1/c^2, i^{-\beta}),
\end{eqnarray}
for all $c\ge 1$, where $c > 0$ is a constant.
Furthermore,   the dependence measures of $\{W_{i,j,k}(u)\}_{i=1}^n$ satisfy  
\begin{eqnarray}
\delta_{W_{j,k}}(h, q)\le \delta_{W_j}(h,q)+\delta_{W_k}(h,q)\le 4\delta_H(h,q)\le Ch^{-\beta}
\end{eqnarray}
 uniformly in $j,k$,
where the second inequality is derived from the fact that
\begin{eqnarray*}
\delta_{W_j}(h,q)&=& \sup_u \Big \|X_h(u)-\int_{0}^1X_h(u)\,du- [X^*_h(u)-\int_{0}^1X^*_h(u)\,du]\Big \|_q\cr
&\le&\sup_u\|X_h(u)-X^*_h(u)\|_q+\int_{0}^1\|X_h(u)-X^*_h(u)\|_q\,du\cr
&\le& 2\sup_u\|X_h(u)-X^*_h(u)\|_q=2\delta_H(h,q).
\end{eqnarray*}
Recall that $\mathring\bW_i(u)$ is the $r(r-1)/2$-dimensional vector with entries $W_{i,j,k} (u)$. The proof of Proposition \ref{prop:tp} now follows from
 applying the same arguments as those given in the proof  of Theorem \ref{thm:1} and Proposition \ref{prop:scb} to  $\mathring\bW_i(u)$.
\end{proof}

\begin{proof}[\bf  Proof of Proposition \ref{power_tp}]
Recall the definition of $\mathring S_{n,j,k}$ in (\ref{h9}) and define
\begin{eqnarray*}
T^*=\max_{1\le j<k\le r}\sup_{0\le u\le 1}|\mathring S_{n,j,k}(u)-\E\mathring S_{n,j,k}(u)|/v_{j,k}(u).
\end{eqnarray*}
Then following the proof of Propositions \ref{prop:tp} and \ref{boots_tp} it follows
\begin{eqnarray}\label{eq:tp1}
\p(T^*> \mathring c_{1-\alpha})\rightarrow 1-\alpha.
\end{eqnarray}
Meanwhile, from the proofs of Propositions \ref{prop:tp} and \ref{boots_tp} we observe that $T^*$ is asymptotically equivalent to the $L^\infty$ norm of a $r(r-1)/2$ dimensional Gaussian random function and the Orcliz norm of each Gaussian random function is $O(1)$. Hence it is straightforward to conclude by a maximal inequality that the Orcliz norm of the latter $L^\infty$ norm is of order $O(\sqrt{\log n})$ in view of $r\asymp n^{\theta_1}$. Therefore $\mathring c_{1-\alpha}=O_\p(\sqrt{\log n})$. On the other hand, it is straightforward to conclude that, under ${\bf H}_a$,
\begin{eqnarray*}
\max_{1\le j<k\le r}\sup_{u\in[0,1]}|\E\mathring S_{n,j,k}(u)|/v_{j,k}(u)|/\sqrt{\log n}\rightarrow\infty.
\end{eqnarray*}
Together with equation (\ref{eq:tp1}), we conclude that $\p (T/\sqrt{\log n}>d_n)\rightarrow 1$ for any divergent sequence $d_n$. Hence the result follows in view of the fact that $\mathring c_{1-\alpha}=O_\p(\sqrt{\log n})$.
\end{proof}

\end{document}